\documentclass[11pt,letterpaper]{amsart}

\makeatletter
\def\subsection{\@startsection{subsection}{2}%
  \z@{.7\linespacing\@plus.7\linespacing}{.2\linespacing}%
  {\centering\normalfont\scshape}}
\makeatother

\usepackage[letterpaper,top=3.8cm, bottom=3.8cm, left=3.7cm, right=3.7cm]{geometry}

\usepackage{amssymb}
\usepackage{enumerate}
\usepackage{url}
\usepackage{mathrsfs}

\newtheorem{theorem}{Theorem}[section]
\newtheorem{proposition}[theorem]{Proposition}
\newtheorem{lemma}[theorem]{Lemma}

\numberwithin{equation}{section}

\newcommand{\CC}{\mathbb{C}}
\newcommand{\FF}{\mathbb{F}}

\newcommand{\ZZ}{\mathbb{Z}}

\newcommand{\Q}{\mathscr{Q}}
\renewcommand{\S}{\mathscr{S}}
\newcommand{\X}{\mathscr{X}}

\newcommand{\Cc}{\mathcal{C}}

\renewcommand{\d}{\delta}
\renewcommand{\t}{\tau}
\newcommand{\e}{\epsilon}

\newcommand{\abs}[1]{\lvert#1\rvert}
\newcommand{\floor}[1]{\lfloor#1\rfloor}
\newcommand{\ang}[1]{\langle#1\rangle}
\newcommand{\sums}[1]{\sum_{\substack{#1}}}

\DeclareMathOperator{\rad}{rad}
\DeclareMathOperator{\rank}{rank}
\DeclareMathOperator{\tr}{tr}
\DeclareMathOperator{\Tr}{Tr}
\DeclareMathOperator{\wt}{wt}
\DeclareMathOperator{\GL}{GL}

\allowdisplaybreaks[3]

\begin{document}

\title[Quadratic and symmetric bilinear forms over finite fields]{Quadratic and symmetric bilinear forms over finite fields and their association schemes}

\author{Kai-Uwe Schmidt}
\address{Department of Mathematics, Paderborn University, Warburger Str.~100, 33098 Paderborn, Germany}
\email[K.-U. Schmidt]{kus@math.upb.de}

\date{11 March 2018}


\begin{abstract}
Let $\Q(m,q)$ and $\S(m,q)$ be the sets of quadratic forms and symmetric bilinear forms on an $m$-dimensional vector space over $\FF_q$, respectively. The orbits of $\Q(m,q)$ and $\S(m,q)$ under a natural group action induce two translation association schemes, which are known to be dual to each other. We give explicit expressions for the eigenvalues of these association schemes in terms of linear combinations of generalised Krawtchouk polynomials, generalising earlier results for odd $q$ to the more difficult case when $q$ is even. We then study $d$-codes in these schemes, namely subsets $X$ of $\Q(m,q)$ or $\S(m,q)$ with the property that, for all distinct $A,B\in X$, the rank of $A-B$ is at least $d$. We prove tight bounds on the size of $d$-codes and show that, when these bounds hold with equality, the inner distributions of the subsets are often uniquely determined by their parameters. We also discuss connections to classical error-correcting codes and show how the Hamming distance distribution of large classes of codes over $\FF_q$ can be determined from the results of this paper.
\end{abstract}

\maketitle


\section{Introduction}

Let $q$ be a prime power and let $V=V(m,q)$ be an $m$-dimensional $\FF_q$-vector space. Let $\Q=\Q(m,q)$ be the space of quadratic forms on $V$ and let $\S=\S(m,q)$ be the space of symmetric bilinear forms on $V$. These spaces are naturally equipped with a metric induced by the rank function. The main motivation for this paper is to study $d$-codes in $\Q$ and $\S$, namely subsets $X$ of $\Q$ or $\S$ such that, for all distinct~$A,B\in X$, the rank of $A-B$ is at least $d$. We are in particular interested in the largest cardinality of $d$-codes in $\Q$ and $\S$ and in the structure of such sets when this maximum is attained. One of the applications is that $d$-codes in $\Q$ can be used to construct optimal subcodes of the second-order generalised Reed-Muller code and our theory can be used to determine the Hamming distance distributions of such codes.
\par
For odd~$q$, most of the results in this paper have been obtained by the author in~\cite{Sch2015}. The new results of this paper concern the more difficult case that $q$ is even, although whenever possible we aim for a unified treatment of the two cases.  For even~$q$, some partial results were obtained previously by the author in~\cite{Sch2010}.
\par
The main tool for studying subsets of $\Q$ and $\S$ is the beautiful theory of association schemes. It is known that $\Q(m,q)$ and $\S(m,q)$ carry the structure of a translation association scheme with $\floor{3m/2}$ classes. These have been studied by Wang, Wang, Ma, and Ma~\cite{WanWanMaMa2003}. In particular the two schemes are dual to each other and, for odd~$q$, they are isomorphic. Hence the association scheme on $\Q$ is self-dual for odd $q$. These association schemes differ considerably from the classical association schemes typically studied by coding theorists, in the sense that the association schemes on~$\Q$ and~$\S$ are neither $P$-polynomial, nor $Q$-polynomial. This means that their most important parameters, namely the $P$- and $Q$-numbers (also known as the first and second eigenvalues), do not just arise from evaluations of sets of orthogonal polynomials. For odd~$q$, the $Q$-numbers (and the $P$-numbers by self-duality) of $\S$ have been determined by the author in~\cite{Sch2015}. For even $q$, the computation of these numbers appears to be more difficult and, so far, only very limited partial results are known. Recursive formulae were given by Feng, Wang, Ma, and Ma~\cite{FenWanMaMa2008} and some special cases were computed by Bachoc, Serra, and Z\'emor~\cite{BacSerZem2017}. The $P$-numbers of $\Q(m,2)$ have been determined by Hou~\cite{Hou2001} using an interesting coding-theoretic approach, which implicitly identifies $\Q(m,2)$ with $R(2,m)$ and $\S(m,2)$ with $R(m,m)/R(m-3,m)$, where $R(r,m)$ is the Reed-Muller code of order~$r$ and length $2^m$. The main result of the present paper is the determination of the $P$- and $Q$-numbers of $\Q$ and, by duality, the $Q$- and $P$-numbers of $\S$. Although these numbers do not directly arise from evaluations of orthogonal polynomials, they can be expressed in terms of linear combinations of generalised Krawtchouk polynomials. This also simplifies the expressions given by Hou~\cite{Hou2001}.
\par
Using the $Q$-numbers of $\Q$ and $\S$, we then obtain tight bounds on the size of~$d$-codes in $\Q$ and $\S$, except in $\S$ when $d$ is even and (by self-duality) in $\Q$ when $d$ is even and $q$ is odd, and give explicit expressions for the inner distributions of $d$-codes when these bounds are attained. In the remaining cases, we obtain tight bounds for $d$-codes that are subgroups $(\Q,+)$ or $(\S,+)$.
\par
In the final section of this paper we briefly discuss the connection between $d$-codes in $\Q$ and classical error-correcting codes. 
It turns out that error-correcting codes obtained from maximal $d$-codes in~$\Q$ often compare favourably to the best known codes. We show that the Hamming distance enumerators of these error-correcting codes are uniquely determined by their parameters. This at once gives the distance enumerators of large classes of error-correcting codes 
for which many special cases have been obtained previously using different methods, for example results for (extended) binary cyclic codes obtained by Berlekamp~\cite{Ber1970} and Kasami~\cite{Kas1971}, recent results for $q$-ary cyclic codes obtained by Li~\cite{Li2017}, and many results for $q$-ary cyclic codes and odd $q$, as explained in~\cite{Sch2015}. In particular, Li~\cite{Li2017} recently determined the true minimum distance of some narrow-sense primitive BCH codes and obtained their distance enumerators in the case that $q$ is odd. These results are recovered in this paper and the distance enumerators are obtained for all $q$ as corollaries of our results.


\section{Quadratic forms and symmetric bilinear forms}

In this section we recall the definitions and basic properties of the association schemes of symmetric bilinear forms and quadratic forms from~\cite{WanWanMaMa2003}. We refer to~\cite{Del1973},~\cite{DelLev1998}, and~\cite{BanIto1984} for more background on association schemes and to~\cite[Chapter~21]{MacSlo1977} and~\cite[Chapter~30]{vLiWil2001} for gentle introductions.
\par
A (symmetric) \emph{association scheme} with $n$ classes is a pair $(\X,(R_i))$, where $\X$ is a finite set and $R_0,R_1,\dots,R_n$ are nonempty relations on $\X$ satisfying:
\begin{enumerate}
\item[(A1)] $\{R_0,R_1,\dots,R_n\}$ is a partition of $\X\times\X$ and $R_0=\{(x,x):x\in \X\}$;
\item[(A2)] Each of the relations $R_i$ is symmetric;
\item[(A3)] If $(x,y)\in R_k$, then the number of $z\in \X$ such that $(x,z)\in R_i$ and $(z,y)\in R_j$ is a constant $p^k_{ij}$ depending only on $i$, $j$, and $k$, but not on the particular choice of $x$ and~$y$.
\end{enumerate}
\par
Let $(\X,(R_i))$ be a symmetric association scheme with $n$ classes and let $D_i$ be the adjacency matrix of the graph $(X,R_i)$. The vector space generated by $D_0,D_1,\dots,D_n$ over the real numbers has dimension $n+1$ and is in fact an algebra, called the \emph{Bose-Mesner algebra} of the association scheme. There exists another uniquely defined basis for this vector space, consisting of minimal idempotent matrices $E_0,E_1,\dots,E_n$. We may write
\[
D_i=\sum_{k=0}^nP_i(k)E_k\quad\text{and}\quad E_k=\frac{1}{\abs{X}}\sum_{i=0}^nQ_k(i)D_i
\]
for some uniquely determined numbers $P_i(k)$ and $Q_k(i)$, called the \emph{$P$-numbers} and the \emph{$Q$-numbers} of $(\X,(R_i))$, respectively.
\par
Now let $V=V(m,q)$ be an $m$-dimensional $\FF_q$-vector space. We denote by 
$\Q=\Q(m,q)$ the set of quadratic forms on $V$ and by $\S=\S(m,q)$ the set of symmetric bilinear forms on $V$. Notice that $\Q$ and $\S$ are themselves $\FF_q$-vectors spaces of dimension $m(m+1)/2$.
\par
Let $G=G(m,q)$ be the direct product $\FF_q^*\times\GL_m(\FF_q)$. Then $G$ acts on~$\Q$ by $(g,Q)\mapsto Q^g$, where $Q^g$ is given by $Q^g(x)=aQ(Lx)$ and $g=(a,L)$. The semidirect product $\Q\rtimes G$ acts transitively on $\Q$ by
\[
((A,g),Q)\mapsto Q^g+A.
\]
The action of $\Q\rtimes G$ extends to $\Q\times\Q$ componentwise and partitions $\Q\times \Q$ into orbits, which define the relations of a symmetric association scheme. Two pairs of quadratic forms $(Q,Q')$ and $(R,R')$ are in the same relation if and only if there is a $g\in G$ such that $(Q-Q')^g=R-R'$. This shows that the relation containing $(Q,Q')$ depends only on $Q-Q'$, which is the defining property of a translation scheme~\cite[Chapter V]{DelLev1998}.
\par
The group $G$ also acts on~$\S$ by $(g,S)\mapsto S^g$, where $S^g$ is given by $S^g(x,y)=aS(Lx,Ly)$ and $g=(a,L)$. The semidirect product $\S\rtimes G$ acts transitively on $\S$ by
\[
((A,g),S)\mapsto S^g+A.
\]
Again, the action of $\S\rtimes G$ extends to $\S\times\S$ componentwise and so partitions $\S\times \S$ into orbits, which define the relations of a symmetric association scheme. Two pairs of symmetric bilinear forms $(S,S')$ and $(T,T')$ are in the same relation if and only if there is a $g\in G$ such that $(S-S')^g=T-T'$, which again makes the association scheme a translation scheme.
\par
When $q$ is odd, every quadratic form $Q\in\Q$ gives rise to a symmetric bilinear form $S\in\S$ via
\begin{equation}
S(x,y)=Q(x+y)-Q(x)-Q(y),   \label{eqn:S_from_Q}
\end{equation}
from which we can recover $Q$ by $Q(x)=\tfrac{1}{2}S(x,x)$. This shows that the association schemes on $\Q$ and $\S$ are isomorphic when $q$ is odd. We shall see that this is not the case when $q$ is even.
\par
Now let $\{\alpha_1,\alpha_2,\dots,\alpha_m\}$ be a basis for $V(m,q)$. For every quadratic form $Q\in\Q$, there exist $A_{ij}\in\FF_q$ such that
\begin{equation}
Q\bigg(\sum_{i=1}^mx_i\alpha_i\bigg)=\sum_{i,j=1}^mA_{ij}x_ix_j   \label{eqn:coordinate_rep_Q}
\end{equation}
for all $(x_1,x_2,\dots,x_m)\in\FF_q^m$. We say that the right hand side is the \emph{coordinate representation} of $Q$ (with respect to the basis chosen). The matrix $A=(A_{ij})$ is only unique modulo the subgroup of $m\times m$ alternating matrices over $\FF_q$. Accordingly we associate with $Q$ the coset $[A]$ of alternating matrices containing $A$.
\par
Let $\{\beta_1,\beta_2,\dots,\xi_m\}$ be another basis for $V(m,q)$. For every symmetric bilinear form $S\in\S$ we then have
\[
S\bigg(\sum_{i=1}^mx_i\beta_i,\sum_{j=1}^my_j\beta_j\bigg)=\sum_{i,j=1}^mB_{ij}x_iy_j,
\]
for all $(x_1,x_2,\dots,x_m)\in\FF_q^m$, where $B_{ij}=S(\beta_i,\beta_j)$. Again, we refer to the right hand side as the \emph{coordinate representation} of $S$. We associate with every symmetric bilinear form the corresponding $m\times m$ symmetric matrix $B=(B_{ij})$.
\par
Let $\chi:\FF_q\to\CC^*$ be a fixed nontrivial character of $(\FF_q,+)$. Hence, if $q=p^k$ for a prime $p$ and an integer $k$, then $\chi(y)=\omega^{\Tr(\theta y)}$ for some fixed $\theta\in\FF_q^*$ and some fixed primitive complex $p$-th root of unity $\omega$. Here, $\Tr:\FF_q\to\FF_p$ is the \emph{absolute trace function} on~$\FF_q$ defined by
\[
\Tr(y)=\sum_{i=1}^ky^{p^i}.
\]
For $Q\in\Q$ and $S\in\S$, write
\begin{equation}
\ang{Q,S}=\chi(\tr(AB)),   \label{eqn:inner_product}
\end{equation}
where $[A]$ is the coset of alternating matrices associated with $Q$ and $B$ is the symmetric matrix associated with $S$ and $\tr$ is the matrix trace. Note that $\ang{Q,S}$ is well defined since $\tr(CD)=0$ if $C$ is alternating and $D$ is symmetric. It is readily verified that $\ang{\,\cdot\,,S}$ ranges through all characters of $(\Q,+)$ when $S$ ranges over $\S$ and that $\ang{Q,\,\cdot\,}$ ranges through all characters of $(\S,+)$ when $Q$ ranges over~$\Q$. Notice that this correspondence depends on the choice of the bases.
\par
The following duality result was observed in~\cite{WanWanMaMa2003}.
\enlargethispage{1.3ex}
\begin{proposition}{\cite[Proposition~3.2]{WanWanMaMa2003}}
\label{pro:duality}
For every $g\in G$ with $g=(a,L)$, we have
\[
\ang{Q^g,S}=\ang{Q,S^h},
\]
where $h=(a,L^T)$.
\end{proposition}
\par
Proposition~\ref{pro:duality} shows that the association schemes on $\Q$ and $\S$ are dual to each other in the strong sense of~\cite[Definition~11]{DelLev1998}.
\par
In what follows, we shall describe the relations of $\Q$ and $\S$ explicitly. For a symmetric bilinear form $S\in\S$, the \emph{radical} is defined to be 
\[
\rad(S)=\{x\in V:\text{$S(x,y)=0$ for every $y\in V$}\}.
\]
The \emph{rank} of $S$ is the codimension of the radical and coincides with the rank of the symmetric matrix associated with $S$. For a quadratic form $Q\in\Q$, let $S_Q(x,y)=Q(x+y)-Q(x)-Q(y)$ be the associated symmetric bilinear form, and define the \emph{radical} of $Q$ to be 
\[
\rad(Q)=\{x\in \rad(S_Q):Q(x)=0\}.
\]
The \emph{rank} of $Q$ is defined to be the codimension of its radical.
\par
The following result describes the orbits of the action of $G$ on $\Q$ and was essentially obtained by Dickson~\cite{Dic1958}.
\begin{proposition}
\label{pro:canonical_quad_forms}
The action of $G$ on $\Q(m,q)$ partitions $\Q(m,q)$ into $\floor{3m/2}+1$ orbits, one of them contains just the zero form. There is one orbit for each odd rank~$r$ and one representative is, in coordinate representation, 
\[
\sum_{i=1}^{(r-1)/2}x_{2i-1}x_{2i}+x_r^2.
\]
There are two orbits for each nonzero even rank $r$ and representatives from the two orbits are, in coordinate representation,
\begin{gather}
\sum_{i=1}^{r/2}x_{2i-1}x_{2i},   \label{eqn:Q_hyperbolic}\\
\sum_{i=1}^{r/2-1}x_{2i-1}x_{2i}+Q_0,   \label{eqn:Q_elliptic}
\end{gather}
where
\[
Q_0=\begin{cases}
x_{r-1}^2+x_{r-1}x_r+\alpha x_r^2  & \text{for $q$ even}\\
x_{r-1}^2-\beta x_r^2              & \text{for $q$ odd}
\end{cases}
\]
and $\alpha\in\FF_q^*$ is a fixed element satisfying $\Tr(\alpha)=1$ (for even $q$) and $\beta\in\FF_q^*$ is a fixed nonsquare in $\FF_q^*$ (for odd $q$). 
\end{proposition}
\par
If $Q$ belongs to an orbit corresponding to~\eqref{eqn:Q_hyperbolic}, then $Q$ is called \emph{hyperbolic} or of \emph{type} $1$ and if~$Q$ belongs to an orbit corresponding to~\eqref{eqn:Q_elliptic}, then $Q$ is called \emph{elliptic} or of \emph{type} $-1$. By convention, the zero form is a hyperbolic quadratic form of rank $0$. Let $\Q_{2s+1}$ be the set of quadratic forms on $V$ of rank ${2s+1}$ and let $\Q_{2s,1}$ and $\Q_{2s,-1}$ be the sets of hyperbolic and elliptic quadratic forms on $V$ of rank $2s$, respectively. Write
\[
I=\left\{2s+1:s\in\ZZ\right\}\cup \left\{(2s,\t):s\in\ZZ,\t=\pm 1\right\}.
\]
and, for every $i\in I$, define the relations 
\begin{equation}
R_i=\{(Q,Q')\in\Q\times \Q:Q-Q'\in \Q_i\}.   \label{eqn:relations_Q}
\end{equation}
The nonempty relations are then precisely the relations of the association scheme of quadratic forms. 
\par
The following result describes the orbits of the action of $G$ on $\S$ and was essentially obtained by Albert~\cite{Alb1938} (and, for odd $q$, also follows from Proposition~\ref{pro:canonical_quad_forms} via~\eqref{eqn:S_from_Q}).
\begin{proposition}
\label{pro:canonical_sym_bil_forms}
The action of $G$ on $\S(m,q)$ partitions $\S(m,q)$ into $\floor{3m/2}+1$ orbits, one of them contains just the zero form. There is one orbit for each odd rank~$r$ and one representative is, in coordinate representation, 
\[
\sum_{i=1}^{(r-1)/2}(x_{2i-1}y_{2i}+x_{2i}y_{2i-1})+x_ry_r.
\]
There are two orbits for each nonzero even rank $r$ and representatives from the two orbits are, in coordinate representation,
\begin{gather}
\sum_{i=1}^{r/2}(x_{2i-1}y_{2i}+x_{2i}y_{2i-1}),   \label{eqn:S_hyperbolic}\\
\sum_{i=1}^{r/2-1}(x_{2i-1}y_{2i}+x_{2i}y_{2i-1})+S_0,   \label{eqn:S_elliptic}
\end{gather}
where 
\[
S_0=\begin{cases}
x_{r-1}y_r+x_ry_{r-1}+x_ry_r & \text{for even $q$}\\
x_{r-1}y_{r-1}-\beta x_ry_r       & \text{for odd $q$},
\end{cases}
\]
and $\beta$ is a fixed nonsquare of $\FF_q^*$ (for odd $q$).
\end{proposition}
\par
Let $\S_{2s+1}$ be the set of symmetric bilinear forms on $V$ of rank $2s+1$ and let $\S_{2s,1}$ and $\S_{2s,-1}$ be the sets of symmetric bilinear forms on $V$ of rank $2s$ corresponding to the orbits~\eqref{eqn:S_hyperbolic} and~\eqref{eqn:S_elliptic}, respectively. Symmetric bilinear forms in $\S_{2s,\t}$ are said to be of \emph{type} $\t$. For even $q$, it can be shown~\cite{Alb1938} that $\S_{2s,1}$ contains precisely the alternating bilinear forms of rank $2s$. For every $i\in I$, define the relations 
\begin{equation}
R'_i=\{(S,S')\in\S\times \S:S-S'\in \S_i\}.   \label{eqn:relations_S}
\end{equation}
The nonempty relations are then precisely the relations of the association scheme of symmetric bilinear forms. 
\par
Now write $v_i=\abs{\Q_i}$ and $\mu_i=\abs{\S_i}$, whose nonzero values are called the \emph{valencies} of the association schemes on $\Q$ and $\S$, respectively. The numbers $v_i$ have been determined by McEliece~\cite{McE1969}, following the work of Dickson~\cite{Dic1958}. Since the association schemes on $\Q$ and $\S$ are isomorphic for odd $q$, we have $\mu_i=v_i$ for odd $q$. For even~$q$, the numbers~$\mu_i$ were determined by MacWilliams~\cite{Mac1969}. We summarise the results in the following form.
\begin{proposition}
\label{pro:valencies_multiplicities}
We have
\begin{align*}
v_{2s+1}&=\mu_{2s+1}=\frac{1}{q^s}\;\frac{\prod\limits_{i=0}^{2s}(q^m-q^i)}{\prod\limits_{i=0}^{s-1}(q^{2s}-q^{2i})},\\[1ex]
v_{2s,\t}&=\frac{q^s+\t}{2}\;\frac{\prod\limits_{i=0}^{2s-1}(q^m-q^i)}{\prod\limits_{i=0}^{s-1}(q^{2s}-q^{2i})},\\[1ex]
\mu_{2s,\t}&=\big(\alpha_\t q^s+\t\beta_sq^{-s}\big)\;\frac{\prod\limits_{i=0}^{2s-1}(q^m-q^i)}{\prod\limits_{i=0}^{s-1}(q^{2s}-q^{2i})},
\end{align*}
where
\[
\alpha_\t=\begin{cases}
\tfrac{1}{2}(1-\t) & \text{for even $q$}\\[1ex]
\tfrac{1}{2}       & \text{for odd $q$}
\end{cases}
\quad\text{and}\quad
\beta_s=\begin{cases}
1                  & \text{for even $q$}\\[1ex]
\tfrac{1}{2}\, q^s & \text{for odd $q$}.
\end{cases}
\]
\end{proposition}
\par
We conclude this section by noting that our association scheme on $\S$ is slightly different from the association schemes on $\S$ in~\cite{WanWanMaMa2003} and~\cite{Sch2015}. The difference is that in~\cite{WanWanMaMa2003} and~\cite{Sch2015} the group $G$ is just $\GL_m(\FF_q)$, which increases the number of orbits from $\floor{3m/2}+1$ to $2m+1$ in the case that $q$ is odd. Another difference to~\cite{Sch2015} is that the sets $\S_{2s,1}$ and $\S_{2s,-1}$, and so also the relations $R'_{2s,1}$ and $R'_{2s,-1}$ on $\S$, are interchanged when $s$ is odd and $q\equiv 3\pmod 4$.


\section{Computation of the $Q$- and $P$-numbers}

Throughout this section we identify quadratic forms with the corresponding cosets of alternating matrices and symmetric bilinear forms with the corresponding symmetric matrices. For $A,B\in\FF_q^{m\times m}$, we write
\[
\ang{A,B}=\chi(\tr(AB)),
\]
where $\chi$ is the same nontrivial character as in~\eqref{eqn:inner_product}. 
The $Q$-numbers and the $P$-numbers of the association scheme on $\Q$ are given by the character sums (see~\cite[Section V]{DelLev1998}, for example)
\begin{align}
Q_k(i)&=\sum_{B\in\S_k}\ang{A,B}\quad \text{for $[A]\in\Q_i$},   \label{eqn:Q_char_sum}\\
P_i(k)&=\sum_{[A]\in\Q_i}\ang{A,B}\quad \text{for $B\in\S_k$},   \label{eqn:P_char_sum}
\end{align}
respectively, where $k,i\in I$. The $Q$-numbers $Q'_i(k)$ and the $P$-numbers $P'_k(i)$ of the association scheme on $\S$ satisfy
\[
Q'_i(k)=P_i(k)\quad\text{and}\quad P'_k(i)=Q_k(i),
\]
respectively. For convenience, we define $Q_k(i)=0$ if $\Q_i=\emptyset$ and $P_i(k)=0$ if $\S_k=\emptyset$. 
\par
In order to give explicit expressions for these numbers, it is convenient to use \emph{$q^2$-analogs of binomial coefficients}, which are defined by
\[
{n\brack k}=\prod_{i=0}^{k-1}\frac{q^{2n}-q^{2i}}{q^{2k}-q^{2i}}
\]
for integral $n$ and $k$. These numbers satisfy the following identities
\begin{equation}
{n\brack k}=q^{2k}{n-1\brack k}+{n-1\brack k-1}={n-1\brack k}+q^{2(n-k)}{n-1\brack k-1}.   \label{eqn:Pascal_triangle}
\end{equation}
We also need the following numbers, which can be derived from generalised Krawtchouk polynomials~\cite{DelGoe1975},~\cite{Del1976}. We define 
\[
F^{(m)}_r(s)=\sum_{j=0}^r(-1)^{r-j}q^{(r-j)(r-j-1)}{n-j\brack n-r}{n-s\brack j}\, c^j,
\]
where
\[
n=\left\lfloor m/2\right\rfloor,\quad\text{and}\quad c=q^{m(m-1)/(2n)},
\]
whenever this expression is defined and let $F^{(m)}_r(s)=0$ otherwise. Equivalently, these numbers can be defined via the $n+1$ equations 
\begin{equation}
\sum_{r=0}^j{n-r\brack n-j} F^{(m)}_r(s)={n-s\brack j}c^j\quad\text{for $j\in\{0,1,\dots,n\}$}   \label{eqn:ev_transform}
\end{equation}
(see~\cite[(29)]{DelGoe1975}).
\par
The following theorem contains explicit expressions for the $Q$-numbers of $\Q(m,q)$. For odd $q$, this is follows from~\cite[Theorem~2.2]{Sch2015}. For even $q$, the result is new.
\begin{theorem}
\label{thm:Q_numbers}
The $Q$-numbers of the association scheme of quadratic forms $\Q(m,q)$ are as follows. We have $Q_{0,1}(i)=1$ and $Q_k(0,1)=\mu_k$ for all $i,k\in I$ and the other $Q$-numbers are given by
\begin{align*}
Q_{2r+1}(2s+1)&=-q^{2r}F^{(m-1)}_r(s),\\
Q_{2r+1}(2s,\t)&=-q^{2r}F^{(m-1)}_r(s-1)+\tau\, q^{m-s+2r}F^{(m-2)}_r(s-1),\\
Q_{2r,\e}(2s+1)&=\alpha_\e\, q^{2r}F^{(m-1)}_r(s)+\e\,\beta_r\,F^{(m)}_r(s),\\
Q_{2r,\e}(2s,\t)&=\alpha_\e\,[q^{2r}F^{(m-1)}_r(s-1)-\tau\, q^{m-s+2r-2}F^{(m-2)}_{r-1}(s-1)]+\e\,\beta_r\,F^{(m)}_r(s),
\end{align*}
where $\alpha_\e$ and $\beta_r$ are given in Proposition~\ref{pro:valencies_multiplicities}.
\end{theorem}
\par
It is well known (and can be easily verified) that the $P$-numbers of $\Q$ can be computed from the $Q$-numbers of $\Q$ via
\[
P_i(k)=\frac{v_i}{\mu_k}Q_k(i).
\]
Proposition~\ref{pro:valencies_multiplicities} then shows that $P_i(k)=Q_k(i)$ for odd $q$ (as it should since the association scheme on $\Q$ is isomorphic to its dual in this case). For even $q$, the $P$-numbers of $\Q$ are given in the following theorem.
\begin{theorem}
\label{thm:P_numbers}
For even $q$, the $P$-numbers of the association scheme of quadratic forms $\Q(m,q)$ are as follows. We have $P_{0,1}(k)=1$ and $P_i(0,1)=v_i$ for all $i,k\in I$ and the other $P$-numbers are given by
\begin{align*}
P_{2s+1}(2r+1)&=-q^{2s}F^{(m-1)}_s(r),\\
2P_{2s,\t}(2r+1)&=q^{2s}F^{(m-1)}_s(r)+\t\, q^s F^{(m)}_s(r),\\
2P_{2s,\t}(2r,1)&=q^s(q^s+\t)F^{(m)}_s(r),\\
2P_{2s,\t}(2r,-1)&=q^{2s}F^{(m-1)}_s(r-1)+\t\,q^sF^{(m)}_s(r),\\
P_{2s+1}(2r,1)&=(q^m-q^{2s})F^{(m)}_s(r),\\
P_{2s+1}(2r,-1)&=-q^{2s}F^{(m-1)}_s(r-1).
\end{align*}
\end{theorem}
\par
In the remainder of this section, we shall prove Theorems~\ref{thm:Q_numbers} and~\ref{thm:P_numbers}. We begin with the following result, which is essentially known. 
\begin{proposition}
\label{pro:Q_alt}
Let $\alpha_\e$ and $\beta_r$ be as in Proposition~\ref{pro:valencies_multiplicities}. The $Q$-numbers of the association scheme $\Q(m,q)$ satisfy
\begin{align*}
\beta_rF^{(m)}_r(s)&=\alpha_{-1}Q_{2r,1}(2s,\t)-\alpha_1Q_{2r,-1}(2s,\t)\\
&=\alpha_{-1}Q_{2r,1}(2s+1)-\alpha_1Q_{2r,-1}(2s+1).
\end{align*}
\end{proposition}
\begin{proof}
For odd $q$, the statement in the lemma can be deduced from~\cite[Lemma~6.3]{Sch2015}, so assume that $q$ is even. Then $\S_{2r,1}$ is the set of alternating bilinear forms of rank~$2r$ on $V$. By Proposition~\ref{pro:canonical_quad_forms}, every quadratic form in $\Q_{2s+1}$ can be represented by an $m\times m$ block diagonal matrix with the block $(1)$ in the top left corner, followed by $s$ copies of
\begin{equation}
\begin{bmatrix}
0 & 1\\
0 & 0
\end{bmatrix}.   \label{eqn:antidiagonal_matrix}
\end{equation}
It can be shown using Proposition~\ref{pro:canonical_quad_forms} that the quadratic form $x^2+xy+\lambda y^2$ in $\Q(2,q)$ is of type $(-1)^{\Tr(\lambda)}$. Hence a quadratic form in $\Q_{2s,\t}$ can be represented by the zero matrix or by an $m\times m$ block diagonal matrix with the block
\begin{equation}
\begin{bmatrix}
\lambda & 1\\
0       & 1
\end{bmatrix}\label{eqn:top_left_block}
\end{equation}
in the top left corner, followed by $s-1$ copies of~\eqref{eqn:antidiagonal_matrix}, where $(-1)^{\Tr(\lambda)}=\t$. It follows from these observations that, for $k=(2r,1)$ the character sums~\eqref{eqn:Q_char_sum} have been evaluated by Delsarte and Goethals~\cite[Appendix]{DelGoe1975}, which gives 
\[
Q_{2r,1}(2s,\t)=Q_{2r,1}(2s+1)=F^{(m)}_r(s),
\]
as required.
\end{proof}
\par
In what follows we write, for every $i\in I$,
\[
Q_{2r}(i)=Q_{2r,1}(i)+Q_{2r,-1}(i).
\]
We shall also write $Q^{(m)}_k(i)$ for $Q_k(i)$ and $\S_k^{(m)}$ for $\S_k$ whenever we need to indicate dependence on $m$.
\par
We have the following recurrences for the $Q$-numbers.
\begin{lemma}
\label{lem:recurrence_Q}
For $k\ge 1$ and $s\ge 0$, we have
\begin{equation}
Q^{(m)}_k(2s+1)=Q^{(m)}_k(2s,1)-q^{m-s}\,Q^{(m-1)}_{k-1}(2s,1)   \label{eqn:rec_Q_1}
\end{equation}
and for $k\ge 1$ and $s\ge 1$, we have
\begin{equation}
Q^{(m)}_k(2s,\t)=Q^{(m)}_k(2s-1)+\tau\,q^{m-s}\,Q^{(m-1)}_{k-1}(2s-1)   \label{eqn:rec_Q_2}.
\end{equation}
\end{lemma}
\begin{proof}
For odd $q$, the lemma can be deduced from~\cite[Lemma~6.1]{Sch2015}, so henceforth we assume that $q$ is even.
\par
To prove the identity~\eqref{eqn:rec_Q_1}, fix an integer $s$ with $0\le s\le (m-1)/2$ and let $A$ be the $m\times m$ block diagonal matrix with the block $(1)$ in the top left corner, followed by $s$ copies of~\eqref{eqn:antidiagonal_matrix}. Then $A$ is a matrix of a quadratic form of rank $2s+1$. Let $A'$ be the $(m-1)\times (m-1)$ matrix obtained from~$A$ by deleting the first row and the first column. Then we have
\begin{align}
Q^{(m)}_k(2s,1)-Q^{(m)}_k(2s+1)&=\sum_{B\in\S^{(m)}_k}(\ang{A',B'}-\ang{A,B})   \nonumber\\
&=\sum_{B\in\S^{(m)}_k}\ang{A',B'}(1-\chi(a)),   \label{eqn:diff_Q_1}
\end{align}
where $\chi$ is the nontrivial character of $(\FF_q,+)$ used to define the pairing in~\eqref{eqn:inner_product} and we write $B$ as
\begin{equation}
B=\begin{bmatrix}
a & u^T\\
u & B'
\end{bmatrix}   \label{eqn:matrix_B}
\end{equation}
for some $a\in\FF_q$, some $u\in\FF_q^{m-1}$ and some $(m-1)\times (m-1)$ matrix $B'$ over~$\FF_q$. The summand in~\eqref{eqn:diff_Q_1} is zero for $a=0$, so assume that $a\ne 0$. Writing
\[
L=\begin{bmatrix}
1 & -a^{-1}u^T\\
0 & I
\end{bmatrix},
\]
we have
\[
L^TBL=\begin{bmatrix}
a & 0\\
0 & C
\end{bmatrix},
\quad\text{where}\quad
C=B'-a^{-1}uu^T.
\]
As $a$ ranges over $\FF_q^*$ and $u$ ranges over $\FF_q^{m-1}$ and $C$ ranges over $\S^{(m-1)}_{k-1}$, the matrix~$B$ in~\eqref{eqn:matrix_B} ranges over $\S^{(m)}_k$ with the constraint $a\ne 0$. Therefore the sum~\eqref{eqn:diff_Q_1} is
\[
\sum_{a\in\FF_q^*}\sum_{u\in\FF_q^{m-1}}\sum_{C\in\S^{(m-1)}_{k-1}}\ang{A',C}\ang{A',a^{-1}uu^T}(1-\chi(a)).
\]
We have
\[
\sum_{C\in\S^{(m-1)}_{k-1}}\ang{A',C}=Q^{(m-1)}_{k-1}(2s),
\]
and
\begin{align*}
\sum_{u\in\FF_q^{m-1}}\ang{A',a^{-1}uu^T}&=q^{m-2s-1}\sum_{u_1,\dots,u_{2s}\in\FF_q}\chi\bigg(a^{-1}\bigg(\sum_{i=1}^su_{2i-1}u_{2i}\bigg)\bigg)\\
&=q^{m-2s-1}\bigg(\sum_{u,v\in\FF_q}\chi(uv)\bigg)^s\\
&=q^{m-s-1}
\end{align*}
for every $a\in\FF_q^*$, and
\[
\sum_{a\in\FF_q^*}(1-\chi(a))=q.
\]
Substitute everything into~\eqref{eqn:diff_Q_1} to obtain the first identity~\eqref{eqn:rec_Q_1} in the lemma.
\par
To prove the identity~\eqref{eqn:rec_Q_2}, fix an integer $s$ with $1\le s\le m/2$ and $\tau\in\{-1,1\}$. Let $\lambda\in\FF_q$ be such that $(-1)^{\Tr(\lambda)}=\t$ and let $A$ be the $m\times m$ block diagonal matrix with the block~\eqref{eqn:top_left_block} in the top left corner, followed by $s-1$ copies of~\eqref{eqn:antidiagonal_matrix}. Then $A$ is a matrix of a quadratic form of rank $2s$ and type $\tau$. Let $A'$ be the $(m-1)\times (m-1)$ matrix obtained from~$A$ by deleting the first row and the first column. Then we have
\begin{equation}
Q^{(m)}_k(2s-1)-Q^{(m)}_k(2s,\tau)=\sum_{B\in\S^{(m)}_k}\ang{A',B'}(1-\chi(a\lambda+u_1)),   \label{eqn:diff_Q_2}
\end{equation}
where we write $B$ as~\eqref{eqn:matrix_B} and where $u=(u_1,\dots,u_{m-1})^T$. We split the summation in~\eqref{eqn:diff_Q_2} into two parts: the sum $S_1$ is over all $B$ with $a\ne 0$ and the sum $S_2$ is over all $B$ with $a=0$. Similarly as in the proof of the first identity~\eqref{eqn:rec_Q_1}, we have
\begin{align}
S_1&=\sum_{a\in\FF_q^*}\sum_{u\in\FF_q^{m-1}}\sum_{C\in\S^{(m-1)}_{k-1}}\ang{A',C}\ang{A',a^{-1}uu^T}(1-\chi(a\lambda+u_1))   \nonumber\\
&=Q^{(m-1)}_{k-1}(2s-1)\sum_{a\in\FF_q^*}\sum_{u\in\FF_q^{m-1}}\ang{A',a^{-1}uu^T}(1-\chi(a\lambda)\chi(u_1)).   \label{eqn:sum_S1}
\end{align}
For every $a\in\FF_q^*$, we have
\begin{align*}
\sum_{u\in\FF_q^{m-1}}\ang{A',a^{-1}uu^T}&=q^{m-2s}\sum_{u_1,\dots,u_{2s-1}\in\FF_q}\chi\bigg(a^{-1}\bigg(u_1^2+\sum_{i=1}^{s-1}u_{2i}u_{2i+1}\bigg)\bigg)\\
&=q^{m-2s}\bigg(\sum_{u,v\in\FF_q}\chi(uv)\bigg)^{s-1}\sum_{w\in\FF_q}\chi(w)\\
&=0
\end{align*}
since the inner sum is zero, and similarly,
\begin{align*}
\sum_{u\in\FF_q^{m-1}}\ang{A',a^{-1}uu^T}\chi(u_1)&=q^{m-2s}\bigg(\sum_{u,v\in\FF_q}\chi(uv)\bigg)^{s-1}\sum_{w\in\FF_q}\chi(a^{-1}w^2+w)\\
&=q^{m-s-1}\sum_{y\in\FF_q}\chi(a(y^2+y)),
\end{align*}
by applying the substitution $w=ay$. The mapping is $y\mapsto y^2+y$ is $2$-to-$1$ and its image is the set of elements in $\FF_q$ whose absolute trace ist zero. Since~$\chi$ is nontrivial, there exists $\theta\in\FF_q^*$ such that
\[
\sum_{y\in\FF_q}\chi(a(y^2+y))=\sum_{y\in\FF_q}(-1)^{\Tr(\theta a(y^2+y))},
\]
which equals $q$ if $a=1/\theta$ and equals zero otherwise. Substitute everything into~\eqref{eqn:sum_S1} to obtain
\[
S_1=-\tau\, q^{m-s}\,Q^{(m-1)}_{k-1}(2s-1),
\]
since $\chi(\lambda/\theta)=(-1)^{\Tr(\lambda)}=\tau$. 
\par
We complete the proof by showing that the sum $S_2$, namely the summation in~\eqref{eqn:diff_Q_2} over all $B$ with $a=0$, equals zero. Let $A''$ be the matrix obtained from $A$ by deleting the first two rows and the first two columns. Then we have
\begin{equation}
S_2=\sums{B\in\S^{(m)}_k\\a=0}\ang{A'',B''}\chi(c)(1-\chi(b)),   \label{eqn:sum_S2}
\end{equation}
where we now write
\[
B=\begin{bmatrix}
E & U^T\\
U & B''
\end{bmatrix}
\quad\text{and}\quad
E=\begin{bmatrix}
a & b\\
b & c
\end{bmatrix}
\]
for some $b,c\in\FF_q$, some $(m-2)\times 2$ matrix $U$ and some $(m-2)\times(m-2)$ matrix~$B''$. Henceforth we put $a=0$. For $b=0$, the summand in~\eqref{eqn:sum_S2} equals zero, so we assume that $b$ is nonzero and so $E$ is invertible. Writing 
\[
M=\begin{bmatrix}
I & -E^{-1}U^T\\
0 & I
\end{bmatrix},
\]
we have
\[
M^TBM=\begin{bmatrix}
E & 0\\
0 & D
\end{bmatrix},
\quad\text{where}\quad
D=B''-UE^{-1}U^T.
\]
Then, arguing similarly as before, we obtain
\begin{align*}
S_2&=\sum_{b\in\FF_q^*}\sum_{c\in\FF_q}\sum_{U\in\FF_q^{(m-2)\times 2}}\sum_{D\in\S^{(m-2)}_{k-2}}\ang{A'',D}\ang{A'',UE^{-1}U^T}\chi(c)(1-\chi(b))\\
&=Q^{(m-2)}_{k-2}(2s-2,1)\sum_{b\in\FF_q^*}\sum_{c\in\FF_q}\sum_{U\in\FF_q^{(m-2)\times 2}}\ang{A'',UE^{-1}U^T}\chi(c)(1-\chi(b)).
\end{align*}
There exists an invertible matrix $2\times 2$ matrix $N$ over $\FF_q$ such that $NE^{-1}N^T$ is either the $2\times 2$ identity matrix or
\[
F=\begin{bmatrix}
0 & 1\\
1 & 0
\end{bmatrix},
\]
depending on whether $c\ne 0$ or $c=0$, respectively. It is readily verified that
\begin{align*}
\sum_{U\in\FF_q^{(m-2)\times 2}}\ang{A'',UU^T}&=\sum_{U\in\FF_q^{(m-2)\times 2}}\ang{A'',UFU^T}\\
&=q^{2(m-2s)}\bigg(\sum_{u,v\in\FF_q}\chi(uv)\bigg)^{2s-2}\\
&=q^{2(m-s-1)}.
\end{align*}
Therefore we have
\[
S_2=q^{2(m-s-1)}\,Q^{(m-2)}_{k-2}(2s-2,1)\sum_{b\in\FF_q^*}(1-\chi(b))\sum_{c\in\FF_q}\chi(c)=0,
\]
since the inner sum is zero. This completes the proof of the identity~\eqref{eqn:rec_Q_2}.
\end{proof}
\par
We shall now solve the recurrence relations in Lemma~\ref{lem:recurrence_Q} using the initial values
\begin{align}
Q^{(m)}_{0,1}(i)&=1   \label{eqn:initial_1}
\intertext{for each $i\in I$ with $\Q_i\ne\emptyset$,}
Q^{(m)}_{2r}(0,1)&=\mu^{(m)}_{2r,1}+\mu^{(m)}_{2r,-1},   \label{eqn:initial_2}\\
Q^{(m)}_{2r+1}(0,1)&=\mu^{(m)}_{2r+1}   \label{eqn:initial_3}
\end{align}
for each $r\ge 0$, where we write $\mu^{(m)}_k$ for $\mu_k$. These initial values follow immediately from~\eqref{eqn:Q_char_sum}.
\begin{proposition}
\label{pro:Q_sum}
For $k\ge 1$, the numbers $Q_k(2s+1)$ for $s\ge 0$ and the numbers $Q_k(2s,\tau)$ and $s\ge 1$ satisfy
\begin{align}
Q_{2r+1}(2s+1) &=-q^{2r}F^{(m-1)}_r(s),                                                    \label{eqn:QS1}\\
Q_{2r+1}(2s,\t)&=-q^{2r}F^{(m-1)}_r(s-1)+\t\, q^{m-s+2r}F^{(m-2)}_r(s-1),                  \label{eqn:QS2}\\
Q_{2r}(2s+1)   &=\phantom{-}q^{2r}F^{(m-1)}_r(s),                                          \label{eqn:QS3}\\
Q_{2r}(2s,\t)  &=\phantom{-}q^{2r}F^{(m-1)}_r(s-1)-\t\, q^{m-s+2r-2}F^{(m-2)}_{r-1}(s-1).  \label{eqn:QS4}
\end{align}
\end{proposition}
\begin{proof}
For odd $q$, the statements in the lemma are given by~\cite[Lemma~6.2]{Sch2015}. However we prove the lemma for odd and even $q$ simultaneously. Write 
\[
n=\floor{(m-1)/2}\quad\text{and}\quad c=q^{(m-1)(m-2)/(2n)}.
\]
From~\eqref{eqn:rec_Q_1} with $s=0$ and the initial values~\eqref{eqn:initial_2} and~\eqref{eqn:initial_3} we have
\begin{align*}
Q^{(m)}_{2r}(1)  &=\mu^{(m)}_{2r,1}+\mu^{(m)}_{2r,-1}-q^m\mu^{(m-1)}_{2r-1},\\
Q^{(m)}_{2r+1}(1)&=\mu^{(m)}_{2r+1}-q^m\big[\mu^{(m-1)}_{2r,1}+\mu^{(m-1)}_{2r,-1}\big].
\end{align*}
From Proposition~\ref{pro:valencies_multiplicities} we then find that
\[
Q^{(m)}_{2r}(1)=-Q^{(m)}_{2r+1}(1)=\frac{1}{q^r}\;\frac{\prod\limits_{i=1}^{2r}(q^m-q^i)}{\prod\limits_{i=0}^{r-1}(q^{2r}-q^{2i})},
\]
which we can write as
\[
q^{2r}{n\brack r}\prod_{j=0}^{r-1}(c-q^{2j}).
\]
This latter expression equals $q^{2r}F^{(m-1)}_r(0)$ (see~\cite{DelGoe1975} or~\cite{Sch2015}, for example) and therefore~\eqref{eqn:QS1} and~\eqref{eqn:QS3} hold for $s=0$. Using the initial value~\eqref{eqn:initial_1}, we see that~\eqref{eqn:QS3} also holds for $r=0$. Now substitute~\eqref{eqn:rec_Q_2} into~\eqref{eqn:rec_Q_1} to obtain
\[
Q^{(m)}_k(2s+1)=Q^{(m)}_k(2s-1)-cq^{2(n-s+1)}\,Q^{(m-2)}_{k-2}(2s-1).
\]
Using~\eqref{eqn:Pascal_triangle}, we verify by induction that~\eqref{eqn:QS1} and~\eqref{eqn:QS3} hold for all $r,s\ge 0$. The identities~\eqref{eqn:QS2} and~\eqref{eqn:QS4} then follow~\eqref{eqn:QS1} and~\eqref{eqn:QS3} and the recurrence~\eqref{eqn:rec_Q_2}. 
\end{proof}
\par
Theorem~\ref{thm:Q_numbers} now follows directly from Propositions~\ref{pro:Q_sum} and~\ref{pro:Q_alt}.
\par
We shall now determine the $P$-numbers of $\Q(m,q)$ for even $q$ from the $Q$-numbers and thereby prove Theorem~\ref{thm:P_numbers}. We begin with stating the $Q$-numbers of $\Q(m,q)$ in the following alternative form.
\begin{proposition}
\label{pro:Q_numbers_sums}
The $Q$-numbers $Q_k(i)$ of the association scheme of quadratic forms $\Q(m,q)$ satisfy
\begin{align}
F^{(m+1)}_r(s)&=Q_{2r,1}(2s-1)+Q_{2r,-1}(2s-1)+Q_{2r-1}(2s-1)   \label{eqn:Q_sum_1}\\
&=Q_{2r,1}(2s,\t)+Q_{2r,-1}(2s,\t)+Q_{2r-1}(2s,\t),   \nonumber\\[1ex]
\t\,q^{m-s}F^{(m)}_r(s)&=Q_{2r,1}(2s,\t)+Q_{2r,-1}(2s,\t)+Q_{2r+1}(2s,\t),   \label{eqn:Q_sum_2}\\
0&=Q_{2r,1}(2s+1)+Q_{2r,-1}(2s+1)+Q_{2r+1}(2s+1),   \label{eqn:Q_sum_2b}\\[1ex]
\beta_rF^{(m)}_r(s)&=\alpha_{-1}Q_{2r,1}(2s,\t)-\alpha_1 Q_{2r,-1}(2s,\t)   \label{eqn:Q_sum_3}\\
&=\alpha_{-1}Q_{2r,1}(2s+1)-\alpha_1Q_{2r,-1}(2s+1)   \nonumber
\end{align}
where $\alpha_\e$ and $\beta_r$ are as in Proposition~\ref{pro:valencies_multiplicities}.
\end{proposition}
\begin{proof}
This follows from Propositions~\ref{pro:Q_alt} and~\ref{pro:Q_sum} using the identity
\begin{equation}
F^{(m+1)}_{r}(s)=q^{2r}F^{(m-1)}_{r}(s-1)-q^{2r-2}F^{(m-1)}_{r-1}(s-1),   \label{eqn:F_identity}
\end{equation}
which can be proved using~\eqref{eqn:Pascal_triangle}.
\end{proof}
\par
We now use Proposition~\ref{pro:Q_numbers_sums} to prove the following counterpart of Proposition~\ref{pro:Q_numbers_sums} for the $P$-numbers.
\begin{proposition}
\label{pro:P_numbers_sums}
The $P$-numbers of the association scheme of quadratic forms $\Q(m,q)$ satisfy
\begin{align}
F^{(m+1)}_s(r)&=P_{2s,1}(2r-1)+P_{2s,-1}(2r-1)+P_{2s-1}(2r-1)   \label{eqn:P_sum_1}\\
&=P_{2s,1}(2r,\e)+P_{2s,-1}(2r,\e)+P_{2s-1}(2r,\e),   \nonumber\\[1ex]
q^s\,F^{(m)}_s(r)&=P_{2s,1}(2r,\e)-P_{2s,-1}(2r,\e)   \label{eqn:P_sum_2}\\
&=P_{2s,1}(2r+1)-P_{2s,-1}(2r+1),   \nonumber\\[0.3ex]
\frac{\alpha_{-\e}}{\beta_r}\,\e\,q^m\,F^{(m)}_s(r)&=P_{2s,1}(2r,\e)+P_{2s,-1}(2r,\e)+P_{2s+1}(2r,\e),   \label{eqn:P_sum_3}\\[-1ex]
0&=P_{2s,1}(2r+1)+P_{2s,-1}(2r+1)+P_{2s+1}(2r+1),   \label{eqn:P_sum_4}
\end{align}
where $\alpha_\e$ and $\beta_r$ are as in Proposition~\ref{pro:valencies_multiplicities}.
\end{proposition}
\begin{proof}
We use the orthogonality relation
\begin{equation}
\sum_{s=0}^{\floor{m/2}}F^{(m)}_p(s)F^{(m)}_s(r)=q^{m(m-1)/2}\delta_{r,p}   \label{eqn:orthogonality_F}
\end{equation}
(see~\cite[(17)]{DelGoe1975}, for example). This shows that the matrix
\[
F=\Big(F^{(m)}_r(s)\Big)_{0\le r,s\le\floor{m/2}}
\]
is invertible and its inverse is $q^{-m(m-1)/2}F$.
\par
Let $\Q^-_s$ be the set of all quadratic forms in $\Q(m,q)$ of rank $2s$ or $2s-1$. Similarly, let $\S^-_r$ be the set of all symmetric matrices in $\S(m,q)$ of rank $2r$ or $2r-1$. From~\eqref{eqn:Q_sum_1} and~\eqref{eqn:Q_char_sum} we find that, for every $[A]\in\Q^-_s$, we have
\[
F_p^{(m+1)}(s)=\sum_{B\in\S^-_p}\ang{A,B}.
\]
Therefore, letting $B'\in\S^-_r$, we have
\begin{align*}
\sum_{s=0}^{\floor{(m+1)/2}}F^{(m+1)}_p(s)\sum_{[A]\in\Q^-_s}\ang{A,B'}&=
\sum_{B\in\S^-_p}\sum_{[A]\in\Q}\ang{A,B+B'}\\[1ex]
&=q^{m(m+1)/2}\,\delta_{r,p},
\end{align*}
by the orthogonality of characters. From the orthogonality relation~\eqref{eqn:orthogonality_F} we then conclude that 
\[
F^{(m+1)}_s(r)=\sum_{[A]\in\Q^-_s}\ang{A,B'},
\]
which, in view of the character sum representation~\eqref{eqn:P_char_sum} of the $P$-numbers, proves \eqref{eqn:P_sum_1}.  
\par
The other identities can be proved similarly. Let $\Q^+_s$ be the set of all quadratic forms in $\Q(m,q)$ of rank $2s$ or $2s+1$ and let $\S^+_r$ be the set of all symmetric matrices in $\S(m,q)$ of rank $2r$ or $2r+1$. From~\eqref{eqn:Q_sum_2},~\eqref{eqn:Q_sum_2b}, and~\eqref{eqn:Q_char_sum} we see that
\[
\sum_{B\in\S^+_p}\ang{A,B}=\begin{cases}
\t\,q^{m-s}F^{(m)}_p(s) & \text{for $[A]\in\Q_{2s,\t}$}\\
0                       & \text{for $[A]\in\Q_{2s+1}$}.
\end{cases}
\]
Let $B'\in\S_r^+$. We then find that
\begin{align*}
\sum_{s=0}^{\floor{m/2}}F^{(m)}_p(s)\sum_{\t\in\{-1,1\}}\t\,q^{m-s}\sum_{[A]\in\Q_{2s,\t}}\ang{A,B'}&=\sum_{s=0}^{\floor{m/2}}\sum_{[A]\in\Q_s^+}\ang{A,B'}\sum_{B\in\S_p^+}\ang{A,B}\\[1ex]
&=\sum_{B\in\S_p^+}\sum_{[A]\in\Q}\ang{A,B+B'}\\[1ex]
&=q^{m(m+1)/2}\,\delta_{r,p}.
\end{align*}
From~\eqref{eqn:orthogonality_F} we conclude that
\[
\sum_{[A]\in\Q_{2s,1}}\ang{A,B'}-\sum_{[A]\in\Q_{2s,-1}}\ang{A,B'}=q^s\,F^{(m)}_s(r),
\]
which together with~\eqref{eqn:P_char_sum} proves~\eqref{eqn:P_sum_2}.
\par
To prove~\eqref{eqn:P_sum_3}, we invoke~\eqref{eqn:Q_sum_3} and~\eqref{eqn:Q_char_sum} to obtain, for $B'\in\S_{2r,\e}$,
\begin{align*}
\sum_{s=0}^{\floor{m/2}}\beta_pF^{(m)}_p(s)\sum_{[A]\in\Q_s^+}\ang{A,B'}&=\sum_{s=0}^{\floor{m/2}}\sum_{[A]\in\Q_s^+}\ang{A,B'}\sum_{\kappa\in\{-1,1\}}\kappa\,\alpha_{-\kappa}\!\sum_{B\in\S_{2p,\kappa}}\ang{A,B}\\[1ex]
&=\sum_{\kappa\in\{-1,1\}}\kappa\,\alpha_{-\kappa}\sum_{B\in\S_{2p,\kappa}}\sum_{[A]\in\Q}\ang{A,B+B'}\\[1ex]
&=\e\,\alpha_{-\e}\,q^{m(m+1)/2}\,\delta_{r,p},
\end{align*}
which, using~\eqref{eqn:orthogonality_F}, gives
\[
\sum_{[A]\in\Q_s^+}\ang{A,B'}=\frac{\alpha_{-\e}}{\beta_r}\,\e\,q^m\,F^{(m)}_s(r).
\]
Now~\eqref{eqn:P_sum_3} follows from~\eqref{eqn:P_char_sum}. 
To prove~\eqref{eqn:P_sum_4}, we take $B'\in\S_{2r+1}$ and obtain similarly as above,
\[
\sum_{s=0}^{\floor{m/2}}\beta_pF^{(m)}_p(s)\sum_{[A]\in\Q_s^+}\ang{A,B'}=0.
\]
This implies that the inner sum is zero for every $s$ and this gives~\eqref{eqn:P_sum_4}.
\end{proof}
\par
We now complete the proof of Theorem~\ref{thm:P_numbers}, which gives explicit expressions for the $P$-numbers.
\begin{proof}[Proof of Theorem~\ref{thm:P_numbers}]
The $P$-numbers of $\Q(m,q)$ are uniquely determined by Proposition~\ref{pro:P_numbers_sums}. We therefore just need to verify that the $P$-numbers claimed in the theorem satisfy the equations in Proposition~\ref{pro:P_numbers_sums}. The identities~\eqref{eqn:P_sum_2},~\eqref{eqn:P_sum_3}, and~\eqref{eqn:P_sum_4} are trivially satisfied. The identity~\eqref{eqn:P_sum_1} is verified using~\eqref{eqn:F_identity} and
\[
F^{(m+1)}_s(r)=q^{2s}F^{(m)}_s(r)+(q^m-q^{2s-2})F^{(m)}_{s-1}(r).
\]
For even $m$, this last identity can be proved directly using~\eqref{eqn:Pascal_triangle}. For odd $m$, first apply~\eqref{eqn:F_identity} and then~\eqref{eqn:Pascal_triangle}.
\end{proof}


\section{Subsets of quadratic and symmetric bilinear forms}

\subsection{Inner distributions, codes, and designs}

In what follows, let $\X=\X(m,q)$ be either $\Q(m,q)$ or $\S(m,q)$. Accordingly, for $i\in I$, let $\X_i$ be either $\Q_i$ or $\S_i$  and let $(R_i)$ be the corresponding relations on $\X$ defined in~\eqref{eqn:relations_Q} and~\eqref{eqn:relations_S}. Let $X$ be a subset of $\X$ and associate with $X$ the rational numbers 
\[
a_i=\frac{\abs{(X\times X)\cap R_i}}{\abs{X}},
\]
so that $a_i$ is the average number of pairs in $X\times X$ whose difference is contained in~$\X_i$. The sequence of numbers $(a_i)_{i\in I}$ is called the \emph{inner distribution} of~$X$. Let $Q_k(i)$ be the $Q$-numbers of $(\X,(R_i))$. The \emph{dual inner distribution} of~$X$ is the sequence of numbers $(a'_k)_{k\in I}$, where
\begin{equation}
a'_k=\sum_{i\in I}\,Q_k(i)\,a_i.   \label{eqn:def_dual_distribution}
\end{equation}
It is a well known fact of the general theory of association schemes that the numbers~$a'_k$ are nonnegative (see~\cite[Theorem~3]{DelLev1998}, for example).
\par
It is readily verified that the mapping $\rho:X\times X\to\ZZ$, given by
\[
\rho(A,B)=\rank(A-B),
\]
is a distance function on $\X$. Accordingly, given an integer $d$ satisfying $1\le d\le m$, we say that $X$ is a \emph{$d$-code} in $\X$ if $\rank(A-B)\ge d$ for all distinct $A,B\in X$. Alternatively, writing
\[
I_\ell=\{2s-1:s\in\ZZ,\,1\le 2s-1\le \ell\}\cup\{(2s,\pm1):s\in\ZZ,\,2\le 2s\le \ell\},
\]
we can define $X$ to be a $d$-code if
\[
a_i=0\quad\text{for each $i\in I_{d-1}$}.
\]
We say that $X$ is a \emph{$t$-design} if
\[
a'_k=0\quad\text{for each $k\in I_t$}.
\]
A subset $X$ of $\X$ is \emph{additive} if $X$ is a subgroup of $(\X,+)$. Note that the inner distribution $(a_i)_{i\in I}$ of an additive subset $X$ of $\X$ satisfies
\[
a_i=\abs{X\cap \X_i},
\]
for every $i\in I$. The \emph{annihilator} of an additive subset $Y$ of $\Q$ is defined to be
\[
Y^\circ=\{S\in \S:\langle Q,S\rangle=1\;\text{for each $Q\in Y$}\}
\]
and the \emph{annihilator} of an additive subset $Z$ of $\S$ is defined to be
\[
Z^\circ=\{Q\in \Q:\langle Q,S\rangle=1\;\text{for each $S\in Z$}\}.
\]
Note that $(Y^\circ)^\circ=Y$ and $(Z^\circ)^\circ=Z$ and 
\[
\abs{\Q}=\abs{Y}\,\abs{Y^\circ}=\abs{Z}\,\abs{Z^\circ}=\abs{\S}.
\]
The following MacWilliams-type identity is a special case of a general property of association schemes (see~\cite[Theorem~27]{DelLev1998}, for example).
\begin{theorem}
\label{thm:dual}
Let $X$ be an additive subset of $\X$ with inner distribution $(a_i)_{i\in I}$ and dual inner distribution $(a'_k)_{k\in I}$ and let $X^\circ$ be its annihilator with inner distribution $(a^\circ_k)_{k\in I}$. Then we have $\abs{X}a^\circ_k=a'_k$.
\end{theorem}


\subsection{Subsets of symmetric bilinear forms}

In this section, we prove bounds on the size of $d$-codes in $\S$. We begin with the following proposition.
\begin{proposition}
\label{pro:inner_dist_S}
Let $Z$ be a subset of $\S(m,q)$ with inner distribution $(a_i)_{i\in I}$ and dual inner distribution $(a'_k)_{k\in I}$. Write 
\begin{align*}
A_r&=a_{2r,1}+a_{2r,-1}+a_{2r-1}, & 
A'_s&=a'_{2s,1}+a'_{2s,-1}+a'_{2s-1},\\
B_r&=a_{2r,1}+a_{2r,-1}+a_{2r+1},&
B'_s&=a'_{2s,1}+a'_{2s,-1}+a'_{2s+1},\\
C_r&=\frac{\alpha_{-1}}{\beta_r}a_{2r,1}-\frac{\alpha_{1}}{\beta_r}a_{2r,-1},&
C'_s&=q^{-s}(a'_{2s,1}-a'_{2s,-1}),
\end{align*}
where $\alpha_\e$ and $\beta_r$ are given in Proposition~\ref{pro:valencies_multiplicities}. Then we have
\begin{align*}
A'_s&=\sum_rF^{(m+1)}_s(r)A_r,\\
C'_s&=\sum_rF^{(m)}_s(r)B_r,\\
B'_s&=q^m\sum_rF^{(m)}_s(r)C_r.
\end{align*}
\end{proposition}
\begin{proof}
Since the $Q$-numbers $Q_k(i)$ of $\S(m,q)$ are the $P$-numbers $P_i(k)$ of $\Q(m,q)$, the result follows directly from~\eqref{eqn:def_dual_distribution} and Proposition~\ref{pro:P_numbers_sums}.
\end{proof}
\par
The following theorem was obtained in~\cite{Sch2015} in the case that~$q$ is odd and in~\cite{Sch2010} in the case that~$q$ is even and $d$ is odd. The case that $q$ and $d$ are even is new.
\begin{theorem}
\label{thm:bounds_S}
Let $Z$ be a $d$-code in $\S(m,q)$, where $Z$ is required to be additive if $d$ is even. Then
\[
\abs{Z}\le\begin{cases}
q^{m(m-d+2)/2} & \text{for $m-d$ even}\\
q^{(m+1)(m-d+1)/2} & \text{for $m-d$ odd}.
\end{cases}
\]
Moreover, in the case of odd $d$, equality occurs if and only if $Z$ is a $t$-design for
\[
t=2\left(\left\lfloor\frac{m+1}{2}\right\rfloor-\frac{d-1}{2}\right).
\]
\end{theorem}
\begin{proof}
As remarked above, the only new case arises when $q$ is even. When~$q$ is odd, the theorem was proved in~\cite[Lemmas~3.5 and~3.6]{Sch2015} using the identities for $A'_s$ and~$C'_s$ in Proposition~\ref{pro:inner_dist_S}. Since these do not involve $C_r$ (which is the only quantity in the conclusion of Proposition~\ref{pro:inner_dist_S} that crucially depends on the parity of $q$), the proofs of~\cite[Lemmas~3.5 and~3.6]{Sch2015} carry over verbatim to the case that $q$ is even.
\end{proof}
\par
We call a $d$-code $Y$ in $\S(m,q)$ \emph{maximal} if $d$ is odd and equality holds in Theorem~\ref{thm:bounds_S}. We shall see in Section~\ref{sec:Constructions} that maximal $d$-codes in $\S$ exist for all possible parameters. 
\par
The situation for even $d$ is somewhat mysterious. Theorem~\ref{thm:bounds_S} gives bounds for the largest additive $d$-codes in $\S(m,q)$ in this case and there certainly exist $d$-codes that are larger than the largest possible additive $d$-code~\cite{Sch2016}. For example, the largest additive $2$-code in $\S(3,2)$ has $16$ elements by Theorem~\ref{thm:bounds_S}, whereas the largest $2$-code in $\S(3,2)$ has $22$ elements~\cite{Sch2016}. In fact, this $2$-code is essentially unique and can be constructed by taking the zero matrix together with all $21$ nonalternating $3\times 3$ symmetric matrices of rank $2$. Moreover,~\cite{Sch2016} contains (not necessarily optimal) $d$-codes in $\S(m,q)$ for many small values of $q$, $m$, and even $d$, which are larger than the largest additive $d$-codes in $\S(m,q)$.


\subsection{Subsets of quadratic forms}

In this section, we prove bounds on the size of $d$-codes in $\Q$. We begin with the following counterpart of Proposition~\ref{pro:inner_dist_S}.
\begin{proposition}
\label{pro:inner_dist_Q}
Let $Y$ be a subset of $\Q(m,q)$ with inner distribution $(a_i)_{i\in I}$ and dual inner distribution $(a'_i)_{i\in I}$. Write 
\begin{align*}
A_s&=a_{2s,1}+a_{2s,-1}+a_{2s-1}, & A'_r&=a'_{2r,1}+a'_{2r,-1}+a'_{2r-1},\\
B_s&=a_{2s,1}+a_{2s,-1}+a_{2s+1}, & B'_r&=a'_{2r,1}+a'_{2r,-1}+a'_{2r+1},\\
C_s&=q^{-s}(a_{2s,1}-a_{2s,-1}),  & C'_r&=\frac{\alpha_{-1}}{\beta_r}a'_{2r,1}-\frac{\alpha_{1}}{\beta_r}a'_{2r,-1},
\end{align*}
where $\alpha_\e$ and $\beta_r$ are given in Proposition~\ref{pro:valencies_multiplicities}. Then we have
\begin{align*}
A'_r&=\sum_sF^{(m+1)}_r(s)A_s,\\
C'_r&=\sum_sF^{(m)}_r(s)B_s,\\
B'_r&=q^m\sum_sF^{(m)}_r(s)C_s.
\end{align*}
\end{proposition}
\begin{proof}
This follows directly from~\eqref{eqn:def_dual_distribution} and Proposition~\ref{pro:Q_numbers_sums}.
\end{proof}
\par
In the next theorem, we give bounds for $d$-codes in $\Q$. Since the association schemes on $\Q(m,q)$ and $\S(m,q)$ are isomorphic for odd $q$, the statement of Theorem~\ref{thm:bounds_S} still holds when $Z$ is a $d$-code in $\Q(m,q)$ and $q$ is odd. We therefore give bounds for $d$-codes in $\Q(m,q)$ only for even $q$.
\begin{theorem}
\label{thm:bound_Q}
Let $q$ be even and let $Y$ be a $d$-code in $\Q(m,q)$. Then
\[
\abs{Y}\le\begin{cases}
q^{m(m-d+2)/2}     & \text{for odd $m$ and odd $d$},\\
q^{(m+1)(m-d+1)/2} & \text{for even $m$ and odd $d$},\\
q^{(m-1)(m-d+2)/2} & \text{for even $m$ and even $d$},\\
q^{m(m-d+1)/2}     & \text{for odd $m$ and even $d$}.
\end{cases}
\]
Moreover, in the case of odd $d$, equality occurs if and only if $Y$ is a $t$-design for
\[
t=2\left(\left\lfloor\frac{m+1}{2}\right\rfloor-\frac{d-1}{2}\right).
\]
\end{theorem}
\begin{proof}
Let $(a_i)_{i\in I}$ be the inner distribution of $Y$. First assume that $d$ is odd, say $d=2\d-1$. Let $A_s$ and $A'_r$ be as defined in Proposition~\ref{pro:inner_dist_Q} and put
\[
n=\lfloor (m+1)/2\rfloor\quad\text{and}\quad c=q^{m(m+1)/(2n)}.
\]
From Proposition~\ref{pro:inner_dist_Q} and~\eqref{eqn:ev_transform} we obtain
\[
\sum_{r=0}^{n-\delta+1}{n-r\brack \delta-1}A'_r=c^{n-\delta+1}\sum_{s=0}^n{n-s\brack n-\delta+1}A_s.
\]
Since $A'_0=\abs{Y}$ and $A_0=1$ and $A_s=0$ for $0<s<\delta$, we obtain
\[
\sum_{r=1}^{n-\delta+1}{n-r\brack \delta-1}A'_r={n\brack \delta-1}(c^{n-\delta+1}-\abs{Y}).
\]
Since the numbers $A'_r$ are nonnegative, the left-hand side is nonnegative, and therefore $\abs{Y}\le c^{n-\delta+1}$, as required. Moreover, this inequality is an equality if and only if $A'_1=\cdots=A'_{n-\d+1}=0$, which is equivalent to $Y$ being a $t$-design for $t=2(n-\d+1)$.
\par
Now assume that $d$ is even, say $d=2\delta$. Let $B_s$ and $C'_r$ be as defined in Proposition~\ref{pro:inner_dist_Q} and put
\[
n=\lfloor m/2\rfloor\quad\text{and}\quad c=q^{m(m-1)/(2n)}.
\]
From Proposition~\ref{pro:inner_dist_Q} and~\eqref{eqn:ev_transform} we obtain
\[
\sum_{r=0}^{n-\delta+1}{n-r\brack \delta-1}C'_r=c^{n-\delta+1}\sum_{s=0}^n{n-s\brack n-\delta+1}B_s
\]
and find, similarly as above,
\[
\sum_{r=1}^{n-\delta+1}{n-r\brack \delta-1}C'_r={n\brack \delta-1}(c^{n-\delta+1}-\abs{Y}).
\]
Again, we have $\abs{Y}\le c^{n-\delta+1}$, which completes the proof.
\end{proof}
\par
We call a $d$-code $Y$ in $\Q(m,q)$ \emph{maximal} if equality holds in Theorem~\ref{thm:bound_Q} or in Theorem~\ref{thm:bounds_S}, unless $q$ is odd and $d$ is even. We shall see in Section~\ref{sec:Constructions} that maximal $d$-codes in $\Q$ exist for all possible parameters. 
\par
An interesting situation, in particular from the coding-theoretic viewpoint of Section~\ref{sec:applications}, occurs for $d$-codes in $\Q$, when $d$ and $m$ are even and no difference between distinct elements is hyperbolic of rank $d$. We call such a set an \emph{elliptic} $d$-code.
\begin{theorem}
\label{thm:bound_elliptic_Q}
Let $m$ and $d$ be even and let $Y$ be an elliptic $d$-code in $\Q(m,q)$. Then
\[
\abs{Y}\le q^{m(m-d+1)/2}.
\]
Moreover, equality occurs if and only if $Y$ is a $t$-design for $t=m-d+1$.
\end{theorem}
\begin{proof}
Write $\delta=d/2$ and $n=m/2$. Let $(a_i)_{i\in I}$ be the inner distribution of~$Y$ and let $A_s$, $A'_r$, $C_s$, and $B'_r$ be as defined in Proposition~\ref{pro:inner_dist_Q}. From Proposition~\ref{pro:inner_dist_Q} and~\eqref{eqn:ev_transform} we obtain
\[
\sum_{r=0}^{n-\delta}{n-r\brack \delta}(q^\delta A'_r+B'_r)=q^{(m+1)(n-\delta)}q^\delta\sum_{s=0}^n{n-s\brack n-\delta}(A_s+q^\delta C_s)
\]
and therefore, since $A_s=C_s=0$ for $0<s<\delta$,
\[
\sum_{r=0}^{n-\delta}{n-r\brack \delta}(q^\delta A'_r+B'_r)=q^{(m+1)(n-\delta)}q^\delta\left({n\brack \delta}(A_0+q^\delta C_0)+A_\delta+q^\delta C_\delta\right).
\]
We have
\begin{align*}
A_\delta+q^\delta C_\delta&=a_{2\delta,1}+a_{2\delta,-1}+a_{2\delta-1}+(a_{2\delta,1}-a_{2\delta,-1})\\
&=2a_{2\delta,1}+a_{2\delta-1}\\
&=0
\end{align*}
since $Y$ is an elliptic $(2\delta)$-code. Since $A_0=C_0=1$ and $A'_0=\abs{Y}$ and $B'_0=\abs{Y}+a'_1$, we then obtain
\[
{n\brack \delta}a'_1+\sum_{r=1}^{n-\delta}{n-r\brack \delta}(q^\delta A'_r+B'_r)={n\brack \delta}(1+q^\delta)(q^{(m+1)(n-\delta)}q^\delta-\abs{Y}).
\]
Since the left-hand side is nonnegative, we find that
\[
\abs{Y}\le q^{(m+1)(n-\delta)}q^\delta.
\]
Moreover, equality occurs if and only if $q^\delta A'_r+B'_r=0$ for all $r$ satisfying $1\le r\le n-\delta$, or equivalently if and only if $Y$ is a $t$-design for $t=m-d+1$.
\end{proof}
We call an elliptic $(2\d)$-code $Y$ in $\Q(2n,q)$ \emph{maximal} if equality holds in Theorem~\ref{thm:bound_elliptic_Q}. We shall see in Section~\ref{sec:Constructions} that maximal elliptic $d$-codes in $\Q(m,q)$ exist for all possible parameters. 


\subsection{Inner distributions of maximal codes}

If $Z$ is a subset of $\S(m,q)$ such that the bound in Theorem~\ref{thm:bounds_S} holds with equality, then in many cases~\cite{Sch2010} and~\cite{Sch2015} give explicit expressions for the inner distribution of~$Z$. These results carry over to subsets of $\Q(m,q)$ in the case that $q$ is odd.
\par
In this section we provide explicit expressions for the inner distributions of maximal $d$-codes in $\Q(m,q)$. We note that, once we know Proposition~\ref{pro:inner_dist_Q} for even $q$, the results in this section can be proved with methods that are very similar to those of~\cite[Section~3.3]{Sch2015}. Hence the proofs in this section are sketched only.
\par
Our first result holds for $d$-codes in $\Q(m,q)$, where $d$ is odd.
\begin{theorem}
\label{thm:inner_dist_d_odd}
If $Y$ is a maximal $(2\d+1)$-code in $\Q(2n+1,q)$, then its inner distribution $(a_i)_{i\in I}$ satisfies
\begin{gather*}
a_{2s-1}={n\brack s-1}\sum_{j=0}^{s-\d-1}(-1)^jq^{j(j-1)}{s\brack j}\big(q^{(2n+1)(s-\d-j)}-1\big),\\[1ex]
a_{2s,\t}=\frac{1}{2}q^s\big(q^s+\t\big)\,{n\brack s}\sum_{j=0}^{s-\d-1}(-1)^jq^{j(j-1)}{s\brack j}\big(q^{(2n+1)(s-\d-j)}-1\big)
\end{gather*}
for $s>0$. If $Y$ is a maximal $(2\d+1)$-code in $\Q(2n,q)$, then its inner distribution $(a_i)_{i\in I}$ satisfies
\begin{align*}
a_{2s-1,\t}&=\frac{1}{2}(q^{2s}-1){n\brack s}\sum_{j=0}^{s-\d-1}(-1)^jq^{j(j-1)}{s-1\brack j}q^{(2n+1)(s-\d-j-1)+2j},\\[1ex]
a_{2s,\t}&=\frac{1}{2}{n\brack s}\sum_{j=0}^{s-\d}(-1)^jq^{j(j-1)}{s\brack j}\big(q^{(2n+1)(s-\d-j)+2j}-1\big)\\
&+\frac{\t}{2}\,q^s{n\brack s}\sum_{j=0}^{s-\d-1}(-1)^jq^{j(j-1)}{s\brack j}\big(q^{(2n+1)(s-\d-j)+2(j-s)}-1\big)
\end{align*}
for $s>0$.
\end{theorem}
\begin{proof}
If $Y$ is a maximal $d$-code in $\Q(m,q)$, where $d$ is odd, then Theorems~\ref{thm:bounds_S} and~\ref{thm:bound_Q} imply that $Y$ is a $t$-design for
\[
t=2\left(\left\lfloor\frac{m+1}{2}\right\rfloor-\frac{d-1}{2}\right).
\]
For odd $q$, the theorem is then~\cite[Theorem~3.9]{Sch2015} and its proof relies just on Proposition~\ref{pro:inner_dist_S}. For even $q$, the proof is almost identical if we use Proposition~\ref{pro:inner_dist_Q} instead of Proposition~\ref{pro:inner_dist_S}.
\end{proof}
\par
The next result holds for maximal $d$-codes in $\Q(m,q)$ when $q$ is even and $d$ is even. In this case, the inner distribution is only partially determined. It is not clear whether there exist such $d$-codes with different inner distribution.
\begin{theorem}
If $q$ is even and $Y$ is a maximal $(2\d)$-code in $\Q(m,q)$, then its inner distribution $(a_i)_{i\in I}$ satisfies
\[
a_{2s,1}+a_{2s,-1}+a_{2s+1}={n\brack s}\sum_{j=0}^{s-\d}(-1)^jq^{j(j-1)}{s\brack j}\big(c^{s-\d-j+1}-1\big)
\]
for $s>0$, where $n=\floor{m/2}$ and $c=q^{m(m-1)/(2n)}$.
\end{theorem}
\begin{proof}
Let $B_s$ and $C'_r$ be defined as in Proposition~\ref{pro:inner_dist_Q}, so that
\[
C'_r=\sum_sF^{(m)}_r(s)B_s.
\]
In particular $B_s=a_{2s,1}+a_{2s,-1}+a_{2s+1}$. If $Y$ is a maximal $(2\d)$-code in $\Q(m,q)$, then we conclude from the proof of Theorem~\ref{thm:bound_Q} that $C'_r=0$ for all $r$ satisfying $1\le r\le n-\d+1$. This gives enough equations to solve for the numbers~$B_s$. The solution is given by~\cite[Lemma~3.8]{Sch2015}.
\end{proof}
\par
The final result of this section concerns maximal elliptic $(2\d)$-codes in $\Q(2n,q)$.
\begin{theorem}
\label{thm:inner_dist_elliptic}
If $Y$ is a maximal elliptic $(2\d)$-code in $\Q(2n,q)$, then its inner distribution $(a_i)_{i\in I}$ satisfies
\begin{align*}
a_{2s-1}&=\frac{1}{2}(q^{2s}-1){n\brack s}\sum_{j=0}^{s-\d-1}(-1)^jq^{j(j-1)}{s-1\brack j}\big(q^{2n(s-\d-j-1)}q^{s+j-1}-1\big),\\[1ex]
a_{2s,\t}&=\frac{1}{2}(q^s+\t){n\brack s}\sum_{j=0}^{s-\d}(-1)^jq^{j(j-1)}{s\brack j}\big(q^{2n(s-\d-j)}q^j-\t\big)
\end{align*}
for $s>0$. 
\end{theorem}
\begin{proof}
If $Y$ is a maximal elliptic $(2\d)$-code in $\Q(2n,q)$, then by Theorem~\ref{thm:bound_elliptic_Q} we have $\abs{Y}=q^{2n(n-\d+1/2)}$ and $Y$ is a $(2n-2\d+1)$-design. The proof is then identical to that of the first part of ~\cite[Proposition~3.10]{Sch2015}.
\end{proof}


\section{Constructions}
\label{sec:Constructions}

In this section we provide constructions of maximal $d$-codes in $\S(m,q)$ and $\Q(m,q)$ using field extensions of $\FF_q$. Throughout this section we take $V=\FF_{q^m}$ and use the \emph{relative trace function} $\Tr_m:\FF_{q^m}\to\FF_q$, which is given by
\[
\Tr_m(y)=\sum_{i=1}^my^{q^i}.
\]

\subsection{Canonical representations}

In what follows we give canonical representations of quadratic forms $Q$ and symmetric bilinear forms $S$ on $\FF_{q^m}$ and describe the pairing $\ang{Q,S}$ in terms of these representations.
\begin{theorem}
\label{thm:sym_quad_unique_representation}
Let $Q\in\Q(m,q)$ be a quadratic form and let $S\in\S(m,q)$ be a symmetric bilinear form.
\begin{enumerate}[(1)]
\item If $m$ is odd, say $m=2n-1$, then there exist unique $f_0,\dots,f_{n-1}\in\FF_{q^m}$ and $g_0,\dots,g_{n-1}\in\FF_{q^m}$ such that $Q$ is given by
\[
Q(x)=\sum_{i=0}^{n-1}\Tr_m(f_ix^{q^i+1})
\]
and $S$ is given by
\[
S(x,y)=\Tr_m(g_0xy)+\sum_{i=1}^{n-1}\Tr_m(g_i(xy^{q^i}+x^{q^i}y)).
\]
Moreover, there are $\FF_q$-bases for $\FF_{q^m}$ such that with respect to these bases we have
\[
\ang{Q,S}=\chi\left(\sum\limits_{i=0}^{n-1}\Tr_m(f_ig_i)\right).
\]

\item If $m$ is even, say $m=2n$, then there exist unique $f_0,\dots,f_{n-1}\in\FF_{q^m}$ and $g_0,\dots,g_{n-1}\in\FF_{q^m}$ and $f_n,g_n\in\FF_{q^n}$ such that $Q$ is given by
\[
Q(x)=\sum_{i=0}^{n-1}\Tr_m(f_ix^{q^i+1})+\Tr_n(f_nx^{q^n+1})
\]
and $S$ is given by
\[
S(x,y)=\Tr_m(g_0xy)+\sum_{i=1}^{n-1}\Tr_m(g_i(xy^{q^i}+x^{q^i}y))+\Tr_m(g_nxy^{q^n}).
\]
Moreover, there are $\FF_q$-bases for $\FF_{q^m}$ such that with respect to these bases we have
\[
\ang{Q,S}=\chi\left(\sum\limits_{i=0}^{n-1}\Tr_m(f_ig_i)+\Tr_n(f_ng_n)\right).
\]
\end{enumerate}
\end{theorem}
\par
To prove Theorem~\ref{thm:sym_quad_unique_representation}, we require some notation and a lemma. Given a linearised polynomial $L\in\FF_{q^m}[X]$ of the form
\begin{equation}
L=\sum_{k=0}^{m-1}c_kX^{q^k},   \label{eqn:linearised_poly}
\end{equation}
we associate with $L$ its \emph{Dickson matrix} $D_L$, given by $(D_L)_{1\le i,j\le m}=c_{j-i}^{q^i}$, where the index of $c_k$ is taken modulo $m$. Henceforth the entries of an $m\times m$ matrix~$M$ are denoted by $M_{ij}$, where $1\le i,j\le m$.
\begin{lemma}
\label{lem:lin_poly_to_mat}
Let $L\in\FF_{q^m}[X]$ be the linearised polynomial~\eqref{eqn:linearised_poly}. Let $\{\xi_1,\xi_2,\dots,\xi_m\}$ be an $\FF_q$-basis for $\FF_{q^m}$ and let $M\in\FF_q^{m\times m}$ be given by  $M_{ij}=\Tr_m(\xi_iL(\xi_j))$. Then we have
\[
M=PD_LP^T,
\]
where $P\in\FF_q^{m\times m}$ is given by $P_{ij}=\xi_i^{q^j}$.
\end{lemma}
\begin{proof}
We can write $M=PR$, where $R_{ij}=L(\xi_j)^{q^i}$. For every $x\in\FF_{q^m}$, we have
\[
L(x)^{q^i}=\sum_{k=1}^mc_{k-i}^{q^i}\,x^{q^k},
\]
where the index is taken modulo $m$. We conclude that $R=D_LP^T$, as required.
\end{proof}
\par
We now prove Theorem~\ref{thm:sym_quad_unique_representation}.
\begin{proof}[Proof of Theorem~\ref{thm:sym_quad_unique_representation}]
It is easy to see that the possible choices for the $f_i$'s and the~$g_i$'s yield $q^{m(m+1)/2}$ quadratic forms and $q^{m(m+1)/2}$ symmetric bilinear forms. In order to prove that these are distinct, it is sufficient to show that $Q$ or $S$ is the zero form if and only if the $f_i$'s are all zero or the $g_i$'s are all zero, respectively. For $\S(m,q)$ and odd $m$, this is accomplished by the proof of Theorem~\ref{thm:constr_sym}. The other cases can be proved similarly, which we leave to the reader. This proves the existence and uniqueness of the $f_i$'s and the $g_i$'s.
\par
It remains to prove the expressions for the pairing $\ang{Q,S}$. We present the proof only in the case that $m$ is odd. Slight modifications also give a proof for even~$m$. Let $\{\alpha_1,\alpha_2,\dots,\alpha_m\}$ and $\{\beta_1,\beta_2,\dots,\beta_m\}$ be a pair of dual $\FF_q$-bases for $\FF_{q^m}$, that is
\[
\Tr_m(\alpha_i\beta_j)=\delta_{ij}\quad\text{for all $i,j$}.
\]
We use the former basis to associate cosets of alternating matrices with quadratic forms and the latter to associate symmetric matrices with symmetric bilinear forms. It will be convenient to define the $m\times m$ matrices $U$ and $V$ by $U_{ij}=\alpha_i^{q^j}$ and $V_{ij}=\beta_i^{q^j}$. Notice that the duality of the two involved bases implies $UV^T=I$, and so $U^TV=I$. 
\par
Define the linearised polynomials
\begin{alignat*}{3}
F_0&=f_0X, \quad & F_1&=\sum_{i=1}^{n-1}(f_iX^{q^i}+f_i^{q^{m-i}}X^{q^{m-i}}),\quad  F_2&\;=\sum_{i=1}^{n-1}f_iX^{q^i},\\
G_0&=g_0X, \quad & G_1&=\sum_{i=1}^{n-1}(g_iX^{q^i}+g_i^{q^{m-i}}X^{q^{m-i}}).
\end{alignat*}
Then we have
\[
S(x,y)=\Tr_m(x(G_0(y)+G_1(y)))
\]
and so the matrix $B$ of $S$ is given by $B=B_0+B_1$, where
\begin{align*}
(B_0)_{ij}&=\Tr_m(\beta_iG_0(\beta_j)),\\
(B_1)_{ij}&=\Tr_m(\beta_iG_1(\beta_j)).
\end{align*}
To associate cosets of alternating matrices with quadratic forms, we distinguish the cases that $q$ is odd or even.
\par
For odd $q$, let $A$ be the unique symmetric matrix associated with the quadratic form $Q$. From~\eqref{eqn:coordinate_rep_Q} we find that this matrix is given by 
\begin{align*}
A_{ij}&=\tfrac{1}{2}(Q(\alpha_i+\alpha_j)-Q(\alpha_i)-Q(\alpha_j))\\[1ex]
&=\Tr_m(\alpha_i(F_0(\alpha_j)+\tfrac{1}{2}F_1(\alpha_j))).
\end{align*}
Write $F=F_0+\tfrac{1}{2}F_1$ and $G=G_0+G_1$ and use Lemma~\ref{lem:lin_poly_to_mat} to obtain
\[
\tr(AB)=\tr(UD_FU^TVD_GV^T)=\tr(D_FD_G),
\]
and therefore
\[
\tr(AB)=\sum\limits_{i=0}^{n-1}\Tr_m(f_ig_i),
\]
as required.
\par
For even $q$, let $A'$ be the unique upper triangular matrix associated with $Q$. From~\eqref{eqn:coordinate_rep_Q} we find that this matrix is given by $A'_{ii}=Q(\alpha_i)$ and $A'_{ij}=Q(\alpha_i+\alpha_j)-Q(\alpha_i)-Q(\alpha_j)$ for $i<j$. In fact, it is more convenient to work with a slightly different matrix of $Q$, namely $A=A_0+A_1$, where $A_0$ and $A_1$ are given by 
\begin{align}
(A_0)_{ij}&=\Tr_m(\alpha_iF_0(\alpha_j))   \nonumber\\[1ex]
(A_1)_{ij}&=\begin{cases}
\Tr_m(\alpha_iF_1(\alpha_j)) & \text{for $i<j$}\\
\Tr_m(\alpha_iF_2(\alpha_j)) & \text{for $i=j$}\\
0                            & \text{otherwise}.
\end{cases}   \label{eqn:def_A1}
\end{align}
Notice that $A-A'$ is alternating, which is in fact the off-diagonal part of~$A_0$. Therefore~$A$ and $A'$ represent the same quadratic form. We have
\[
\tr(AB)=\tr(A_0B_0)+\tr(A_1B_1)+\tr(A_0B_1)+\tr(A_1B_0).
\]
Using Lemma~\ref{lem:lin_poly_to_mat} we have
\[
\tr(A_0B_0)=\tr(UD_{F_0}U^TVD_{G_0}V^T)=\tr(D_{F_0}D_{G_0})=\Tr_m(f_0g_0).
\]
Now define an inner product on alternating matrices in $\FF_q^{m\times m}$ by
\[
(X,Y)=\sum_{i<j}X_{ij}Y_{ij}.
\]
This inner product satisfies
\[
(WXW^T,Y)=(X,W^TYW)
\]
for every $W\in\FF_q^{m\times m}$. Using this property and Lemma~\ref{lem:lin_poly_to_mat}, we obtain
\[
\tr(A_1B_1)=(A_1+A_1^T,B_1)=(UD_{F_1}U^T,VD_{G_1}V^T)=(D_{F_1},D_{G_1}),
\]
and therefore
\[
\tr(A_1B_1)=\sum_{i=1}^{n-1}\Tr_m(f_ig_i).
\]
Now, since $A_0$ is symmetric and $B_1$ is alternating, we have $\tr(A_0B_1)=0$. From Lemma~\ref{lem:lin_poly_to_mat}, we find that $B_0=VD_{G_0}V^T$, where $D_{G_0}$ is a diagonal matrix. We claim that $A_1=UEU^T$, where $E$ has only zeros on the main diagonal. This implies that
\[
\tr(A_1B_0)=\tr(UEU^TVD_{G_0}V^T)=\tr(ED_{G_0})=0,
\]
and so completes the proof.
\par
It remains to prove the claim. The required matrix $E$ is given by $E=V^TA_1V$, and so for every $i$, we have using~\eqref{eqn:def_A1}
\begin{align*}
E_{ii}&=\sum_{k,\ell}\beta_k^{q^i}(A_1)_{k\ell}\,\beta_\ell^{q^i}\\
&=\sum_k\beta_k^{q^i}\Tr_m(\alpha_kF_2(\alpha_k))\beta_k^{q^i}+\sum_{k<\ell}\beta_k^{q^i}\Tr_m(\alpha_kF_1(\alpha_\ell))\beta_\ell^{q^i}\\
&=\sum_k\beta_k^{q^i}\Tr_m(\alpha_kF_2(\alpha_k))\beta_k^{q^i}+\sum_{k<\ell}\beta_k^{q^i}\Tr_m(\alpha_kF_2(\alpha_\ell)+\alpha_\ell F_2(\alpha_k))\beta_\ell^{q^i}\\
&=\sum_{k,\ell}\beta_k^{q^i}\Tr_m(\alpha_kF_2(\alpha_\ell))\beta_\ell^{q^i}\\
&=v_i^TUD_{F_2}U^Tv_i=(v_i^TU)D_{F_2}(v_i^TU)^T,
\end{align*}
where $v_i$ is the $i$-th column of $V$. Since $V^TU=I$, we find that the main diagonal of~$E$ equals the main diagonal of $D_{F_2}$, which is zero. This proves the claim.
\end{proof}
\par


\subsection{The constructions}

We now give constructions of maximal $d$-codes in $\S(m,q)$ and $\Q(m,q)$. We begin with recalling constructions from~\cite{Sch2010} and~\cite{Sch2015} of additive $d$-codes in $\S(m,q)$.
\begin{theorem}[{\cite[Theorem.~4.4]{Sch2015}}]
\label{thm:constr_sym}
Let $d$ be an integer with the same parity as $m$ satisfying $1\le d\le m$ and let $Z$ be the subset of $\S(m,q)$ formed by the symmetric bilinear forms
\begin{gather*}
S:\FF_{q^m}\times \FF_{q^m}\to\FF_q\\
S(x,y)=\Tr_m(g_0xy)+\sum_{i=1}^{(m-d)/2}\Tr_m(g_i(xy^{q^i}+x^{q^i}y)),\quad g_i\in\FF_{q^m}.
\end{gather*}
Then $Z$ is an additive $d$-code in $\S(m,q)$ of size $q^{m(m-d+2)/2}$. In particular, $Z$ is a maximal $d$-code in~$\S(m,q)$ for odd $m$ and is maximal among additive $d$-codes in~$\S(m,q)$ for even $m$.
\end{theorem}
\par
Whenever $m-d$ is odd, Theorem~\ref{thm:constr_sym} gives $(d+2)$-codes $Z$ in $\S(m+1,q)$ for which equality holds in Theorem~\ref{thm:bounds_S}. Let $W$ be an $m$-dimensional subspace of $V(m+1,q)$ and define the \emph{punctured} set (with respect to $W$) of $Z$ to be
\[
Z^*=\big\{S|_W:S\in Z\big\},
\] 
where $S|_W$ is the restriction of $S$ onto $W$. Then $Z^*$ is a $d$-code in $\S(m,q)$ for which again equality holds in Theorem~\ref{thm:bounds_S}. This shows that $Z^*$ is a maximal $d$-code in~$\S(m,q)$ for odd $d$ and is maximal among additive $d$-codes in~$\S(m,q)$ for even $d$.
\par
For odd $q$, Theorem~\ref{thm:sym_quad_unique_representation} of course also gives corresponding sets of quadratic forms by associating a quadratic form $Q$ with $S$ via $Q(x)=\tfrac{1}{2}S(x,x)$. It therefore remains to give constructions of maximal $d$-codes in $\Q(m,q)$ for even $q$. The following consequence of Theorems~\ref{thm:sym_quad_unique_representation} and~\ref{thm:constr_sym} gives a construction for $d$-codes in~$\Q(m,q)$ when both $m$ and~$d$ are odd (and where $q$ can have either parity).
\begin{theorem}
\label{thm:constr_m_odd_d_odd}
Let $m$ and $d$ be odd integers satisfying $1\le d\le m$ and let~$Y$ be the subset of $\Q(m,q)$ formed by the quadratic forms
\begin{gather*}
Q:\FF_{q^m}\to\FF_q\\
Q(x)=\sum_{i=(d-1)/2}^{(m-1)/2}\Tr_m(f_ix^{q^i+1}),\quad f_i\in\FF_{q^m}.
\end{gather*}
Then $Y$ is additive and a maximal $d$-code in $\Q(m,q)$ of size $q^{m(m-d+2)/2}$.
\end{theorem}
\begin{proof}
It is plain that $Y$ is additive and has size $q^{m(m-d+2)/2}$. From Theorems~\ref{thm:sym_quad_unique_representation} and~\ref{thm:constr_sym} we find that the annihilator of $Y^\circ$ of $Y$ is a maximal $(m-d+3)$-code in $\S(m,q)$. Theorem~\ref{thm:bounds_S} implies that $Y^\circ$ is a $(d-1)$-design and Theorem~\ref{thm:dual} then implies that~$Y$ is a $d$-code. From Theorem~\ref{thm:bound_Q} we find hat $Y$ is maximal.
\end{proof}
\par
Whenever $d$ is odd and~$m$ is even, Theorem~\ref{thm:constr_m_odd_d_odd} gives maximal $(d+2)$-codes $Y$ in $\Q(m+1,q)$. In fact, Theorem~\ref{thm:bound_Q} implies that $Y$ is also a maximal $(d+1)$-code in $\Q(m+1,q)$. Let $W$ be an $m$-dimensional subspace of $V(m+1,q)$ and define the \emph{punctured} set (with respect to $W$) of $Y$ to be
\[
Y^*=\big\{Q|_W:Q\in Y\big\},
\]
where $Q|_W$ is the restriction of $Q$ onto $W$. Then $Y^*$ is a maximal $d$-code in $\Q(m,q)$. This leaves the case that $m$ and $d$ are both even. In this case we have the following construction, which identifies $V$ with $\FF_{q^{m-1}}\times\FF_q$ and is essentially contained in~\cite{DelGoe1975}.
\begin{theorem}
\label{thm:constr_elliptic_codes}
Let $q$ be even, let $m$ and $d$ be even integers satisfying $1\le d\le m$, and let $Y$ be the subset of $\Q(m,q)$ formed by the quadratic forms
\begin{gather*}
Q:\FF_{q^{m-1}}\times\FF_q\to\FF_q\\
Q(x,u)=\sum_{i=1}^{m/2-1}\Tr_{m-1}\big((f_0x)^{q^i+1}\big)+u\Tr_{m-1}(f_0x)+\sum_{i=1}^{(m-d)/2}\Tr_{m-1}\big(f_ix^{q^i+1}\big),
\end{gather*}
where $f_i\in\FF_{q^{m-1}}$. Then $Y$ is a maximal $d$-code in $\Q(m,q)$ of size $q^{(m-1)(m-d+2)/2}$.
\end{theorem}
\begin{proof}
The quadratic form $Q$ polarises to the bilinear form
\[
\Tr_{m-1}(f_0^2xy+f_0y\Tr_{m-1}(f_0x)+f_0(uy+vx))+\sum_{i=1}^{(m-d)/2}\Tr_{m-1}(f_i(xy^{q^i}+x^{q^i}y)).
\]
It is known~\cite[Theorem~9]{DelGoe1975} that the difference between two such forms for distinct $(f_0,f_1,\dots,f_{(m-d)/2})$ has rank at least $d$. Therefore $Y$ is a $d$-code in $\Q(m,q)$ of size $q^{(m-1)(m-d+2)/2}$, hence a maximal $d$-code in $\Q(m,q)$ by Theorem~\ref{thm:bound_Q}.
\end{proof}
\par
We close this section by giving a construction for maximal elliptic $d$-codes in~$\Q(m,q)$.
\begin{theorem}
\label{thm:constr_elliptic}
Let $m$ be even and write $m=2n$. Let $\delta$ be an integer satisfying $1\le\delta\le n$ and let $Y$ be the subset of $\Q(m,q)$ formed by the quadratic forms
\begin{gather*}
Q:\FF_{q^m}\to\FF_q\\
Q(x)=\sum_{i=\delta}^{n-1}\Tr_m(f_ix^{q^i+1})+\Tr_n(f_nxy^{q^n}),\quad f_i\in\FF_{q^m},f_n\in\FF_{q^n}.
\end{gather*}
Then $Y$ is a maximal elliptic $(2\delta)$-code in $\Q(m,q)$ of size $q^{m(n-\d+1/2)}$.
\end{theorem}
\begin{proof}
It is plain that $Y$ is additive. A straightforward computation gives
\[
Q(x+y)-Q(x)-Q(y)=\Tr_m(yL(x)),
\]
where
\[
L(x)=f_nx^{q^n}+\sum_{i=\delta}^{n-1}\left(f_ix^{q^i}+f_i^{q^{2n-i}}x^{q^{2n-i}}\right).
\]
Since $L(x^{q^{2n-\delta}})$ is induced by a polynomial of degree at most $2n-2\delta$, we find that~$Q$ has rank at least $2\delta$, unless $f_\delta=\dots=f_n=0$. Hence $Y$ is $(2\delta)$-code of size~$q^{m(n-\delta+1/2)}$. 
\par
Let $(a_i)_{i\in I}$ be the inner distribution of $Y$ and let $A_s$, $A'_r$, $C_s$, and $B'_r$ be as defined in Proposition~\ref{pro:inner_dist_Q}. By Theorems~\ref{thm:sym_quad_unique_representation} and~\ref{thm:constr_sym}, the annihilator $Y^\circ$ of $Y$ is a $(2n-2\delta+2)$-code in $\S(m,q)$. Thus Theorem~\ref{thm:dual} implies that $A'_r=B'_r=0$ for all $r$ satisfying $1\le r\le n-\delta$. As in the proof of Theorem~\ref{thm:bound_elliptic_Q}, we find that
\[
{n\brack \delta}(q^\delta A'_0+B'_0)=q^{(m+1)(n-\delta)}q^\delta\left({n\brack \delta}(A_0+q^\delta C_0)+A_\delta+q^\delta C_\delta\right).
\]
Since $A'_0=B'_0=\abs{Y}=q^{m(n-\delta+1/2)}$ and $A_0=C_0=1$, we conclude that
\[
A_\delta+q^\delta C_\delta=0.
\]
We have $A_\delta+q^\delta C_\delta=2a_{2\delta,1}+a_{2\delta-1}$ by definition and $a_{2\delta-1}=0$ since $Y$ is a $(2\delta)$-code. Therefore $a_{2\delta,1}=0$, and so $Y$ is an elliptic $(2\delta)$-code, hence a maximal elliptic $(2\delta)$-code in $\Q(m,q)$ by Theorem~\ref{thm:bound_elliptic_Q}.
\end{proof}


\section{Applications to classical coding theory}
\label{sec:applications}

In this section we construct classical error-correcting codes over finite fields from subsets of $\Q(m,q)$, extending results from~\cite{Sch2015} for odd $q$.
\par
A \emph{code} over $\FF_q$ of \emph{length}~$n$ is a subset of $\FF_q^n$; such a code is \emph{additive} if it is a subgroup of $(\FF_q^n,+)$. The (Hamming) \emph{weight} of $c\in\FF_q^n$, denoted by $\wt(c)$, is the number of nonzero entries in $c$. This weight induces a distance on $\FF_q^n$ and the smallest distance between two distinct elements of a code $\Cc$ is called the \emph{minimum distance} of $\Cc$. 
We associate with a code~$\Cc$ the polynomials
\[
\alpha(z)=\sum_{c\in\Cc}z^{\wt(c)}
\]
and
\[
\beta(z)=\frac{1}{\abs{\Cc}}\sum_{b,c\in\Cc}z^{\wt(c-b)},
\]
which are called the \emph{weight enumerator} and the \emph{distance enumerator} of $\Cc$, respectively. Note that, if~$\Cc$ is additive, then its weight enumerator coincides with its distance enumerator. 
\par
As usual, we let $V=V(m,q)$ be an $m$-dimensional $\FF_q$-vector space and $\Q(m,q)$ the set of quadratic forms on $V$. Since for every quadratic form $Q:V\to\FF_q$ we have $Q(0)=0$, we shall identify functions from $V$ to $\FF_q$ with vectors $\FF_q^{V^*}$, where $V^*=V-\{0\}$.
\par
Let $R_q(1,m)^*$ be the set of all $q^{m+1}$ affine functions from $V$ to $\FF_q$. This code has length $q^m-1$ and is the punctured version of the generalised first-order Reed-Muller code $R_q(1,m)$ of length $q^m$. If we identify $V$ with $\FF_{q^m}$, then $R_q(1,m)^*$ consists of the functions
\begin{gather*}
\FF_{q^m}\to\FF_q\\
x\mapsto\Tr_m(ax)+c,\quad a\in\FF_{q^m},\,c\in\FF_q.
\end{gather*}
We shall associate codes with subsets $Y$ of $\Q(m,q)$ by taking cosets of $R_q(1,m)^*$ with coset representatives from $Y$. Care must be taken in the case that $q=2$ since $x^2=x$ for all $x\in\FF_2$, which implies that every quadratic form in $\Q(m,2)$ of rank~$1$ is in fact also a linear function. Accordingly, we define a subset~$Y$ of $\Q(m,q)$ to be \emph{nondegenerate} if $q>2$ or if $q=2$ and $Y$ contains no forms of rank~$1$. For every nondegenerate subset~$Y$ of~$\Q(m,q)$, we define the code $\Cc(Y)$ of size $q^{m+1}\,\abs{Y}$ by
\[
\Cc(Y)=\bigcup_{Q\in Y}Q+R_q(1,m)^*.
\]
If $Y$ equals $\Q(m,q)$, then $\Cc(Y)$ is the punctured version $R_q(2,m)^*$ of the generalised second-order Reed-Muller code $R_q(2,m)$ of length $q^m$.
\par
For $i\in I$, define the polynomial
\[
\omega_i(z)=n_1z^{w_1}+n_2z^{w_2}+n_3z^{w_3}+n_4z^{w_4}+n_5z^{w_5}+n_6z^{w_6},
\]
where
\begin{align*}
w_1&=q^{m-1}(q-1)-q^{m-s-1}-1 & n_1&=\tfrac{1}{2}(q^{2s}(q-1)-q^s)(q-1)\\
w_2&=q^{m-1}(q-1)-q^{m-s-1}   & n_2&=\tfrac{1}{2}(q^{2s}+q^s)(q-1)\\
w_3&=q ^{m-1}(q-1)-1 & n_3&=(q^m-q^{2s}(q-1))(q-1)\\
w_4&=q ^{m-1}(q-1)   & n_4&=q^m-q^{2s}(q-1)\\
w_5&=q^{m-1}(q-1)+q^{m-s-1}-1 & n_5&=\tfrac{1}{2}(q^{2s}(q-1)+q^s)(q-1)\\
w_6&=q^{m-1}(q-1)+q^{m-s-1}   & n_6&=\tfrac{1}{2}(q^{2s}-q^s)(q-1)
\intertext{for $i=2s+1$ and}
w_1&=(q^{m-1}-\tau q^{m-s-1})(q-1)-1 &  n_1&=(q^{2s-1}-\tau q^{s-1})(q-1)\\
w_2&=(q^{m-1}-\tau q^{m-s-1})(q-1)   &  n_2&=q^{2s-1}+\tau q^{s-1}(q-1)\\
w_3&=q^{m-1}(q-1)-1 & n_3&=(q^m-q^{2s})(q-1)\\ 
w_4&=q^{m-1}(q-1)   & n_4&=q^m-q^{2s}\\
w_5&=q^{m-1}(q-1)+\tau q^{m-s-1}-1 & n_5&=(q^{2s-1}(q-1)+\tau q^{s-1})(q-1)\\
w_6&=q^{m-1}(q-1)+\tau q^{m-s-1}   & n_6&=(q^{2s-1}-\tau q^{s-1})(q-1)
\end{align*}
for $i=(2s,\t)$. The following result relates the polynomial $\omega_i(z)$ with the weight enumerator of cosets of $R_q(1,m)^*$. This result can be proved using the standard theory of quadratic forms. (Recall that $\Q_{2s+1}$ contains all quadratic forms of rank $2s+1$ and $\Q_{2s,\t}$ contains all quadratic forms of rank $2s$ and type $\t$.)
\begin{lemma}[{\cite[Propositions 4.1 and 5.1]{Li2017}}]
\label{lem:coset_weight_distribution}
Let $Q\in\Q(m,q)$ be a quadratic form with $Q\in\Q_i$. Then $\omega_i(z)$ is the weight enumerator of the coset $Q+R_q(1,m)^*$.
\end{lemma}
\par
Now, since $R_q(1,m)^*$ is additive, the distance enumerator of $\Cc(Y)$ equals
\[
\frac{1}{\abs{Y}}\sum_{b,c\in Y}\;\sum_{a\in R_q(1,m)^*}z^{\wt(a+c-b)}.
\]
The inner sum is the weight enumerator of the coset $c-b+R_q(1,m)^*$ and so Lemma~\ref{lem:coset_weight_distribution} gives the distance enumerator of $\Cc(Y)$ in terms of the inner distribution of $Y$.
\begin{theorem}
\label{thm:dist_C2}
Let $Y$ be a nondegenerate subset of $\Q(m,q)$ with inner distribution $(a_i)_{i\in I}$. Then the distance enumerator of $\Cc(Y)$ is $\sum_{i\in I}a_i\omega_i(z)$.
\end{theorem}
\par
If~$Y$ equals $\Q(m,q)$, then Theorem~\ref{thm:inner_dist_d_odd} with $\d=0$ and Theorem~\ref{thm:dist_C2} give the distance enumerator of $R_q(2,m)^*$. This complements results of McEliece~\cite{McE1969}, who determined the distance enumerator of the second-order generalised Reed-Muller code $R_q(2,m)$ itself. This latter result can also be recovered from Theorem~\ref{thm:inner_dist_d_odd} and a slightly modified version of Theorem~\ref{thm:dist_C2}.
\par
Now let $m$ and $d$ be two integers of equal parity satisfying $1\le d\le m$. If $d$ is odd, let~$Y$ be a nondegenerate maximal $d$-code in $\Q(m,q)$ and if $d$ is even, let $Y$ be a maximal elliptic $d$-code in $\Q(m,q)$. Writing $\delta=\floor{d/2}$, we have by Theorems~\ref{thm:bounds_S},~\ref{thm:bound_Q}, and~\ref{thm:bound_elliptic_Q}
\[
\abs{Y}=q^{m(m-2\delta+1)/2}.
\]
The code $\Cc(Y)$ has length $q^m-1$, cardinality $q^{m(m-2\delta+3)/2+1}$, and minimum distance
\begin{equation}
q^{m-1}(q-1)-q^{m-\delta-1}-1.   \label{eqn:designed_distance}
\end{equation}
The distance enumerator of $\Cc(Y)$ is determined by Theorems~\ref{thm:dist_C2} and~\ref{thm:inner_dist_d_odd} for odd~$d$ and by Theorems~\ref{thm:dist_C2} and~\ref{thm:inner_dist_elliptic} for even $d$.
\par
Now assume that $Y$ is obtained from the specific constructions in Theorems~\ref{thm:constr_m_odd_d_odd} and~\ref{thm:constr_elliptic}, according to whether $d$ is odd or even, respectively. Then $\Cc(Y)$ is a linear code and, in many cases, $\Cc(Y)$ is an optimal linear code or has the same parameters as the best known linear code~\cite{Gra2007}. Generalising work of Berlekamp~\cite{Ber1970}, it was shown by Li~\cite[Proposition~2.5]{Li2017} that if
\[
\frac{m}{3}\le \delta\le \frac{m}{2},
\]
then $\Cc(Y)$ is a narrow-sense primitive BCH code of designed minimum distance~\eqref{eqn:designed_distance}. Hence, in this case, the true minimum distance of $\Cc(Y)$ equals its designed minimum distance. This recovers principal results of~\cite{Ber1970} for $q=2$ and of~\cite{Li2017} for odd $q$. Using the results of~\cite{Sch2015} and additional arguments, the distance enumerator of $\Cc(Y)$ was obtained in~\cite{Li2017} for odd $q$. Using entirely different methods, the distance enumerator of the extended version of $\Cc(Y)$ was also obtained for $q=2$ in~\cite{Ber1970}. Our results give, in a uniform way, the distance enumerator of $\Cc(Y)$ for every prime power $q$. 
\par
Berlekamp~\cite{Ber1970} and Kasami~\cite{Kas1971} studied cyclic codes of the form $\Cc(Y)$ and related codes for other specific subsets $Y$ of $\Q(m,2)$. They determined the distance enumerators of such codes using methods that are completely different from our methods. Many of these results can be recovered and generalised to~$q>2$ using Theorems~\ref{thm:inner_dist_d_odd} and~\ref{thm:inner_dist_elliptic} together with Theorem~\ref{thm:dist_C2} or some suitable modification.
\par
We close this section by noting that, if~$Y$ is a maximal $(2\delta)$-code in $\Q(m,q)$ and~$m$ and $q$ are even, then $\Cc(Y)$ has length $q^m-1$, cardinality $q^{m(m-2\d+4)/2-(m-2\d)/2}$, and minimum distance
\[
(q^{m-1}-q^{m-\delta-1})(q-1)-1.
\]
For $q=2$, the extended version of $\Cc(Y)$ is known as the Delsarte-Goethals code and for $2\d=m$ it is known as the Kerdock code~\cite[Ch.~15]{MacSlo1977}.

\section*{Acknowledgment}

I would like to thank Shuxing Li for helpful discussions on applications to error-correcting codes.



\begin{thebibliography}{10}

\bibitem{Alb1938}
A.~A. Albert, \emph{Symmetric and alternate matrices in an arbitrary field.
  {I}}, Trans. Amer. Math. Soc. \textbf{43} (1938), no.~3, 386--436.

\bibitem{BacSerZem2017}
Ch. Bachoc, O.~Serra, and G.~Z\'emor, \emph{An analogue of {V}osper's theorem
  for extension fields}, Math. Proc. Cambridge Philos. Soc. \textbf{163}
  (2017), no.~3, 423--452.

\bibitem{BanIto1984}
E.~Bannai and T.~Ito, \emph{Algebraic combinatorics {I}: Association schemes},
  The Benjamin/Cummings Publishing Co., Inc., Menlo Park, CA, 1984.

\bibitem{Ber1970}
E.~R. Berlekamp, \emph{The weight enumerators for certain subcodes of the
  second order binary {R}eed-{M}uller codes}, Information and Control
  \textbf{17} (1970), 485--500.

\bibitem{Del1973}
Ph. Delsarte, \emph{An algebraic approach to the association schemes of coding
  theory}, Philips Res. Rep. Suppl. \textbf{10} (1973).

\bibitem{Del1976}
\bysame, \emph{Properties and applications of the recurrence
  {$F(i+1,k+1,n+1)=q^{k+1}F(i,k+1,n)-q^{k}F(i,k,n)$}}, SIAM J. Appl. Math.
  \textbf{31} (1976), no.~2, 262--270.

\bibitem{DelGoe1975}
Ph. Delsarte and J.~M. Goethals, \emph{Alternating bilinear forms over {${\rm
  GF}(q)$}}, J. Combin. Theory Ser. A \textbf{19} (1975), no.~1, 26--50.

\bibitem{DelLev1998}
Ph. Delsarte and V.~I. Levenshtein, \emph{Association schemes and coding
  theory}, IEEE Trans. Inform. Theory \textbf{44} (1998), no.~6, 2477--2504.

\bibitem{Dic1958}
L.~E. Dickson, \emph{Linear groups: {W}ith an exposition of the {G}alois field
  theory}, Dover Publications, Inc., New York, 1958.

\bibitem{FenWanMaMa2008}
R.~Feng, Y.~Wang, Ch. Ma, and J.~Ma, \emph{Eigenvalues of association schemes
  of quadratic forms}, Discrete Math. \textbf{308} (2008), no.~14, 3023--3047.

\bibitem{Gra2007}
M.~Grassl, \emph{Bounds on the minimum distance of linear codes and quantum
  codes}, Online available at \url{http://www.codetables.de}, 2007.

\bibitem{Hou2001}
X.-D. Hou, \emph{The eigenmatrix of the linear association scheme on
  {$R(2,m)$}}, Discrete Math. \textbf{237} (2001), no.~1-3, 163--184.

\bibitem{Kas1971}
T.~Kasami, \emph{The weight enumerators for several classes of subcodes of the
  {$2$}nd order binary {R}eed-{M}uller codes}, Information and Control
  \textbf{18} (1971), 369--394.

\bibitem{Li2017}
Sh. Li, \emph{The minimum distance of some narrow-sense primitive {BCH} codes},
  SIAM J. Discrete Math. \textbf{31} (2017), no.~4, 2530--2569.

\bibitem{MacSlo1977}
F.~J. MacWilliams and N.~J.~A. Sloane, \emph{The theory of error-correcting
  codes}, Amsterdam, The Netherlands: North Holland, 1977.

\bibitem{Mac1969}
J.~MacWilliams, \emph{Orthogonal matrices over finite fields}, Amer. Math.
  Monthly \textbf{76} (1969), 152--164.

\bibitem{McE1969}
R.~McEliece, \emph{Quadratic forms over finite fields and second-order
  {R}eed-{M}uller codes}, JPL Space Programs Summary 37-58 \textbf{III} (1969),
  28--33.

\bibitem{Sch2010}
K.-U. Schmidt, \emph{Symmetric bilinear forms over finite fields of even
  characteristic}, J. Combin. Theory Ser. A \textbf{117} (2010), no.~8,
  1011--1026.

\bibitem{Sch2015}
\bysame, \emph{Symmetric bilinear forms over finite fields with applications to
  coding theory}, J. Algebraic Combin. \textbf{42} (2015), no.~2, 635--670.

\bibitem{Sch2016}
M.~Schmidt, \emph{Rank metric codes}, Master's thesis, University of Bayreuth,
  Germany, 2016.

\bibitem{vLiWil2001}
J.~H. van Lint and R.~M. Wilson, \emph{A course in combinatorics}, second ed.,
  Cambridge University Press, Cambridge, 2001.

\bibitem{WanWanMaMa2003}
Y.~Wang, C.~Wang, C.~Ma, and J.~Ma, \emph{Association schemes of quadratic
  forms and symmetric bilinear forms}, J. Algebraic Combin. \textbf{17} (2003),
  no.~2, 149--161.

\end{thebibliography}

\providecommand{\bysame}{\leavevmode\hbox to3em{\hrulefill}\thinspace}
\providecommand{\MR}{\relax\ifhmode\unskip\space\fi MR }
\providecommand{\MRhref}[2]{%
  \href{http://www.ams.org/mathscinet-getitem?mr=#1}{#2}
}
\providecommand{\href}[2]{#2}

\end{document}